\documentclass[11pt]{amsart}
\usepackage{amssymb,amsthm}
\usepackage[nosumlimits]{mathtools}
\usepackage{enumerate}
\usepackage[margin=1in]{geometry}
\usepackage{eucal}
\usepackage{mathrsfs}
\usepackage{stmaryrd}
\usepackage{upgreek}
\usepackage{tabu}
\usepackage{leftidx}
\usepackage{quiver}

\DeclareSymbolFont{sfletters}{OT1}{cmsmf}{m}{n}
\DeclareMathSymbol{\fff}{\mathord}{sfletters}{"49}

\theoremstyle{plain}
\newtheorem{theorem}{Theorem}[section]
\newtheorem{proposition}[theorem]{Proposition}

\newtheorem{corollary}[theorem]{Corollary}
\newtheorem*{theorem*}{Theorem}

\theoremstyle{definition}

\theoremstyle{remark}

\usepackage{hyperref}
\usepackage{xcolor}
\hypersetup{
    colorlinks,
    linkcolor={red!50!black},
    citecolor={blue!50!black},
    urlcolor={blue!80!black}
}

\let\O\undefined
\DeclareMathOperator{\O}{O}
\DeclareMathOperator{\SO}{SO}
\DeclareMathOperator{\GL}{GL}

\DeclareMathOperator{\V}{V}
\DeclareMathOperator{\St}{St}
\DeclareMathOperator{\im}{im}
\DeclareMathOperator{\proj}{proj}
\DeclareMathOperator{\id}{id}
\DeclareMathOperator{\Hom}{Hom}
\DeclareMathOperator{\Skew}{\mathsf{\Lambda}}

\DeclareMathOperator{\tr}{tr}
\DeclareMathOperator{\Gr}{Gr}
\DeclareMathOperator{\diag}{diag}
\DeclareMathOperator{\rank}{rank}

\DeclareMathOperator{\Rie}{\mathsf{Rie}}
\DeclareMathOperator{\Ric}{\mathsf{Ric}}
\DeclareMathOperator{\Sca}{\mathsf{Sca}}
\DeclareMathOperator{\sff}{\fff\fff}
\DeclareMathOperator{\tff}{\fff\fff\fff}
\DeclareMathOperator{\dive}{div}
\DeclareMathOperator{\End}{End}
\DeclareMathOperator{\oGr}{\stackrel{\smash[t]{\sim}}{\smash{\operatorname{Gr}}\rule{0pt}{1.1ex}}}

\newcommand{\grad}[1]{\nabla_{\!#1}}
\newcommand{\lgrad}[1]{\leftidx{_#1\!}{\nabla}}

\newcommand{\quo}[3]{#1(#2)\! \! \bigm/ \! \! \bigl( #1(#3) \times #1(#2 - #3) \bigr)}

\newcommand{\sonk}{\mathfrak{so}(k) \oplus  \mathfrak{so}(n - k)}
\newcommand{\sosonk}{\mathfrak{so}(n) \!\! \bigm/ \!\! \bigl( \mathfrak{so}(k) \oplus  \mathfrak{so}(n - k) \bigr)}

\newcommand{\tp}{{\scriptscriptstyle\mathsf{T}}}

\newcommand{\lb}{\llbracket}
\newcommand{\rb}{\rrbracket}

\let\latexcirc=\circ
\newcommand{\ccirc}{\mathbin{\mathchoice
  {\xcirc\scriptstyle}
  {\xcirc\scriptstyle}
  {\xcirc\scriptscriptstyle}
  {\xcirc\scriptscriptstyle}
}}
\newcommand{\xcirc}[1]{\vcenter{\hbox{$#1\latexcirc$}}}
\let\circ\ccirc

\begin{document}
\title{Simple matrix expressions for the curvatures of Grassmannian}
\author[Z.~Lai]{Zehua~Lai}
\address{Department of Mathematics, University of Texas, Austin, TX 78712}
\email{zehua.lai@austin.utexas.edu}
\author[L.-H.~Lim]{Lek-Heng~Lim}
\address{Computational and Applied Mathematics Initiative, Department of Statistics,
University of Chicago, Chicago, IL 60637-1514}
\email{lekheng@uchicago.edu}
\author[K.~Ye]{Ke Ye}
\address{State Key Laboratory of Mathematical Sciences, Academy of Mathematics and Systems Science, Chinese Academy of Sciences, Beijing 100190, China}
\email{keyk@amss.ac.cn}

\begin{abstract}
We show that modeling a Grassmannian as symmetric orthogonal matrices $\Gr(k,\mathbb{R}^n) \cong\{Q \in  \mathbb{R}^{n \times n} : Q^\tp Q = I, \; Q^\tp = Q,\; \tr(Q)=2k - n\}$ yields exceedingly simple matrix formulas for various curvatures and curvature-related quantities, both intrinsic and extrinsic. These include Riemann, Ricci, Jacobi, sectional, scalar,  mean, principal, and Gaussian curvatures; Schouten, Weyl, Cotton, Bach, Pleba\'nski, cocurvature, nonmetricity, and torsion tensors; first, second, and third fundamental forms; Gauss and Weingarten maps; and upper and lower delta invariants. We will derive explicit, simple expressions for the aforementioned quantities in terms of standard matrix operations that are stably computable with numerical linear algebra. Many of these aforementioned quantities have never before been presented for the Grassmannian.
\end{abstract}

\maketitle

\section{Introduction}\label{sec:intro}

While pure mathematicians typically abhor picking coordinates for manifolds, this is all but inevitable in applied mathematics. A good choice of extrinsic coordinates facilitates computations for the applied mathematician and, as we will see in this article, provides transparent, easy-to-calculate expressions that are useful even for investigations in pure mathematics. 

For the Grassmannian of $k$-planes in $\mathbb{R}^n$, we showed in \cite{ZLK20} that points on the manifold may be represented by matrices $Q \in \mathbb{R}^{n \times n}$ that are (i) symmetric $Q^\tp = Q$, (ii) orthogonal $Q^\tp Q = I$, (iii) involutive $Q^2 = I$. Clearly any two of these conditions imply the third and thus
\begin{equation}\label{eq:inv0}
\Gr(k,\mathbb{R}^n) \cong \{Q \in \mathbb{S}^n : Q^2 =I,\; \tr(Q)=2k - n\} \eqqcolon \Gr(k,n)
\end{equation}
where $\mathbb{S}^n$ denotes the Euclidean space of $n \times n$ symmetric matrices.
Our motivation in \cite{ZLK20} was largely computational --- such a coordinate representation of points in $\Gr(k,\mathbb{R}^n)$ by orthogonal matrices gives immeasurably stabler numerical algorithms compared to other models of the Grassmannian as projection matrices or equivalence classes of matrices.

The goal of this article is to show that, even for calculations by hand, the \emph{involution model} \eqref{eq:inv0} provides a significant advantage over, say, expressions in \cite{Strichartz88}, which are supposedly simple and already given in terms of linear algebra. Henceforth we define $\Gr(k,n)$ to be the set of matrices on the right of \eqref{eq:inv0} to distinguish it from $\Gr(k,\mathbb{R}^n)$, the Grassmannian as an abstract manifold. In our earlier work \cite{ZLK20}, we derived expressions for basic quantities related to optimization:  tangent vector,  normal vector, metric, exponential map, geodesic, parallel transport, gradient, Hessian, etc, and showed that they all have simple, easily computable expressions in the involution model. Here we will do the same for various types of curvatures, some of which are notoriously difficult to calculate, but it is nevertheless rewarding as curvatures are likely the most important geometric objects of a Riemannian manifold. One might even argue that Riemannian geometry was created to provide a rigorous platform for studying curvatures.

A secondary goal is to illustrate the ease of using the involution model \eqref{eq:inv0}. That $\Gr(k, \mathbb{R}^n)$ can be realized as a linear section of $\O(n)$, one of the most familiar and best-studied classical Lie group, is a feature unique to the involution model. We believe that many of the expressions derived in this article would be more difficult, some nearing impossible, to derive in other common models of the Grassmannian --- as submanifolds of projective spaces, as various homogeneous spaces, or as a manifold of orthogonal projectors, all discussed in Section~\ref{sec:other}. Moreover, the expressions that we obtained are also more user-friendly, as we will elaborate below after presenting them in Table~\ref{tab:c}.

\begin{table*}[h]
\footnotesize
\caption{The point $Q \in \Gr(k,n)$, tangent vectors $X,Y,Z,W \in \mathbb{T}_{Q} \Gr(k,n)$, and normal vector $H \in \mathbb{N}_{Q} \Gr(k,n)$ are parameterized as in \eqref{eq:para}.}
\tabulinesep=0.75ex
\begin{tabu}{@{}lll}
\textsc{curvature} & \textsc{expression}  & \textsc{result} \\\tabucline[1pt ] -
first fundamental form & $\fff(X, Y) = 2 \tr(X_0^\tp Y_0)$  & Proposition~\ref{prop:fff} \\
second fundamental form & $\sff(X,Y) =\dfrac{1}{2} V \begin{bmatrix}
-X_0 Y_0^\tp - Y_0 X_0^\tp & 0 \\
0 & X_0^\tp Y_0 + Y_0^\tp X_0
\end{bmatrix} V^\tp$  & Theorem~\ref{thm:2ndFF}\\
third fundamental form & $\tff(X,Y) = \left( \dfrac{n}{2k(n-k)} - \dfrac{n-2}{4} \right) \tr(X_0^\tp Y_0)$  & Corollary~\ref{cor:third}\\
Gauss map & $\mathsf{\Gamma} (Q) = \biggl\lbrace V  
\begin{bmatrix}
H_1 & 0 \\
0 & H_2 
\end{bmatrix}  V^\tp  : H_1 \in \mathbb{S}^k,\; H_2 \in \mathbb{S}^{n-k} \biggr\rbrace$ & Proposition~\ref{prop:Gmap} \\
Weingarten map & $\mathsf{S}(H)(X) = \dfrac12 V
\begin{bmatrix}
0 &  X_0 H_2 - H_1 X_0  \\
( X_0 H_2 - H_1 X_0 )^\tp & 0
\end{bmatrix} V^\tp$  & Corollary~\ref{cor:wein} \\
mean curvature vector & $
\mathsf{H} = \dfrac{1}{2k(n-k)} V\begin{bmatrix}
-(n - k) I_k & 0 \\
0 & k I_{n-k}
\end{bmatrix} V^\tp$  &  Corollary~\ref{cor:meancurv}\\
mean curvature & $\mathsf{H}(H) = \dfrac{ (k - n) \tr H_1 + k \tr H_2 }{2k(n-k)}$  & Corollary~\ref{cor:meancurv}\\
Gaussian curvature & $\mathsf{G}(H) = \dfrac{1}{2^{k(n-k)}}\prod_{i=1}^k \prod_{j=1}^{n-k} (\lambda_{k+j} - \lambda_i)$  & Corollary~\ref{cor:prin}\\
principal curvature & $\upkappa_{ij}(H) = \dfrac{1}{2}(\lambda_{k+j} - \lambda_i)$,\quad  $i=1,\dots, k$, $j = 1,\dots, n-k$  & Corollary~\ref{cor:prin}\\
Riemann curvature & $ \Rie(X,Y,Z, W ) = \dfrac{1}{2}
\tr\bigl( (XY - YX) Z W \bigr)$  & Proposition~\ref{prop:Rcurv}  \\
Jacobi curvature & $ \mathsf{J}(X, Y, Z, W) = \tr(XYZW)  - \tr\Bigl(  Y \Bigl(\dfrac{X Z  + ZX}{2}\Bigr) W  \Bigr) $  & Corollary~\ref{cor:Jcurv}  \\
sectional curvature & $ \upkappa(X,Y) = \dfrac{1}{4} \dfrac{\|XY - YX\|^2}{\lVert X \rVert^2 \lVert Y \rVert^2 - \tr( X  Y )^2}$ & Corollary~\ref{cor:Scurv}\\
Ricci curvature & $\Ric(X,Y) = \dfrac{n-2}{8} \tr(X Y)$  & Corollary~\ref{cor:Ricci} \\
scalar curvature & $\Sca = \dfrac{k(n-k)(n-2)}{8}$  & Corollary~\ref{cor:Ricci}\\
traceless Ricci curvature & $\mathsf{Z}(X,Y) = 0$  & Corollary~\ref{cor:Ricci} \\
upper delta invariant & $\overline{\updelta}_{2,r} = \dfrac{k(n-k)(n-2)}{8}$  & Theorem~\ref{thm:delta}\\
lower delta invariant & $\underline{\updelta}_{2,r} = \dfrac{k(n-k)(n-2)}{8} - \dfrac{r}{4}$  & Theorem~\ref{thm:delta}\\
Schouten curvature & $\mathsf{P}(X,Y) = \dfrac{n-2}{16(k(n-k)-1)} \tr(X Y)$  & Corollary~\ref{cor:Schouten}\\
Cotton curvature & $\mathsf{C}(X,Y,Z) = 0$  & Corollary~\ref{cor:Cotton}\\
Bach curvature & $\mathsf{B}(X,Y) = \dfrac{(n-2)^2}{32(k(n-k) - 2)} \tr(X Y) $  & Corollary~\ref{cor:Bach}  \\
Weyl curvature & $\mathsf{W}(X,Y,Z, W ) =\dfrac{1}{2} \tr \bigl( (XY - YX) Z  W \bigr)$ & Corollary~\ref{cor:Weyl}  \\
& $- \dfrac{(n-2)}{8(k(n-k) - 1)} \bigl( \tr(X Z) \tr(Y W)  - \tr(XW) \tr(Y Z) \bigr)$  &
\end{tabu}
\label{tab:c}
\end{table*}

The precise definitions and conventions in Table~\ref{tab:c} are given in Section~\ref{sec:zoo} and as we will also see therein, every object in this table involves at most one QR decomposition and a handful of matrix multiplications to compute. A few of the curvatures in Table~\ref{tab:c} are extrinsic, i.e., they depend specifically on our model \eqref{eq:inv0}. These include the second and third fundamental forms; the Gauss and Weingarten maps; mean, Gaussian, and principal curvatures. 
They help us better understand the embedded geometry of Grassmannian given by \eqref{eq:inv0} but they also expedite our calculations of the intrinsic curvatures. These include the Riemann, Ricci, sectional, and scalar curvatures; the Schouten, Cotton, Weyl, and Bach tensors; and the upper and lower delta invariants. Extrinsic curvatures in Table~\ref{tab:c} have ambient space $\mathbb{S}^n$ equipped with inner product $\tr(X^\tp Y) = \tr(XY)$.

While the value of an intrinsic curvature is independent of our choice of models, the expression or formula that gives this value is not. As is evident from Table~\ref{tab:c}, the involution model yields simple, stably computable formulas for extrinsic and intrinsic curvatures alike, essentially reducing curvatures of the Grassmannian to matrix analysis and their computations to numerical linear algebra. For example, computing the value of the Riemann curvature, a daunting order-$4$ tensor, is a trivial one-line calculation using our formula (and the proof of this formula is notably also a one-liner). For contrast, we will show  in Section~\ref{sec:quot} what the corresponding calculation would entail if we use the most common Grassmannian model $\quo{\O}{n}{k}$.

\subsection*{What's new}

For the intrinsic curvatures,  the Schouten, Cotton, Weyl, Bach curvature tensors, the upper and lower delta invariants have never been explicitly calculated for a Grassmannian to the best of our knowledge.  The formulas for Riemann,  Ricci,  sectional, and scalar curvatures of the Grassmannian modeled as various quotient spaces (see Section~\ref{sec:quot}) are well-known and classical \cite{Cartan46, do1992, Samelson58, Wong68} and they have also been calculated in  \cite{BZA24,AI21} for the projection model (see Section~\ref{sec:proj}).  The novelty of our calculations for these is that we derived intrinsic curvatures from extrinsic curvatures. This is why we will calculate extrinsic curvatures first.

The formulas for the extrinsic curvatures --- second and  third fundamental forms; Gauss and Weingarten maps; mean, Gaussian, and principal curvatures ---  are all new. Of course this is just a consequence of the relative obscurity of the involution model \eqref{eq:inv0}. Unlike intrinsic invariants, extrinsic ones are  model-dependent, and it is expected that these have never been calculated for a new model. Also, extrinsic curvatures do not apply to the quotient models in Section~\ref{sec:quot} as they are only defined for embedded manifolds.

We emphasize that by `formula' we mean an explicit expression like those in Table~\ref{tab:c}, involving actual matrices, and has no undetermined quantities. These curvatures may of course be expressed in terms of \emph{local coordinates} or \emph{equivalence classes} or \emph{horizontal spaces}, but these invariably require additional computational overhead, which we will discuss in Section~\ref{sec:other}. Our formulas do not contain ambiguities that require further choices and effort to resolve.

In addition to the curvatures in Table~\ref{tab:c}, we will also discuss the cocurvature, nonmetricity, torsion, and Pleba\'nski tensors. We will see in Proposition~\ref{prop:zero} and Corollary~\ref{cor:Pleb} that they are trivially zero. We also proved in Corollary~\ref{cor:nullity} that the index of relative nullity vanishes; in Corollary~\ref{cor:Codazzi} that the third fundamental form, the Ricci curvature, the Schouten and Bach tensors are all Codazzi tensors;  and in Corollary~\ref{cor:divfree} that the Riemann and Weyl curvatures are divergence-free.

Although the goal of our article is not pure mathematical investigations of the Grassmannian, the  values of these curvature as well as the involution model itself can nevertheless be useful towards this end. The current record for the geodesic embedding of $2$-spheres into a Grassmannian is that $\Gr(k,\mathbb{R}^n)$ contains \emph{one} geodesic $2$-sphere \cite{Wong61}. We will show in Theorem~\ref{thm:packing}, reproduced below for easy reference, that this result can be vastly improved:
\begin{theorem*}[Embedding products of geodesic $2$-spheres]
\begin{enumerate}[\upshape (a)]
\item For any $r \le \min\{ \lfloor k/2 \rfloor,  \lfloor n/4 \rfloor \}$, the product of $r$ copies of $\mathrm{S}^2$ can be embedded as a totally geodesic submanifold of $\Gr(k,\mathbb{R}^n)$.
\item For any $r \le 2\lfloor k/2 \rfloor \lfloor (n-k)/2 \rfloor$, an open subset of the product of $r$ copies of $\mathrm{S}^2$ can be embedded as a totally geodesic submanifold  of $\Gr(k,\mathbb{R}^n)$.
\end{enumerate}
\end{theorem*}
As the reader will see, the proof of the second statement specifically exploits the matrix structure of the involution model \eqref{eq:inv0}.

\section{Notations and conventions}\label{sec:note}

In this article, we use blackboard bold fonts for vector spaces (e.g., tangent and normal spaces, space of $n \times n$ symmetric matrices, etc.) and san serif fonts for all curvatures and curvature-related quantities (e.g., Table~\ref{tab:c}). The Riemann, Ricci, and scalar curvatures, arguably the three most important quantities, are given three-letter notations $\Rie, \Ric, \Sca$ for emphasis.

We write $\mathbb{E}^m$ for a Euclidean space of dimension $m$ equipped with its Euclidean inner product $\langle \, \cdot, \cdot \, \rangle$. For concreteness, one may assume that this Euclidean space is $\mathbb{R}^n$ with $\langle x,y\rangle = x^\tp y$, or $\mathbb{S}^n$ with $\langle X,Y\rangle = \tr(XY)$, or $\mathbb{R}^{m \times n}$ with $\langle X,Y\rangle = \tr(X^\tp Y)$.  We write $\proj_{\mathbb{W}} : \mathbb{E}^m \to \mathbb{E}^m$ for the orthogonal projection onto a subspace $\mathbb{W}  \subseteq \mathbb{E}^m$. The space of all linear maps between vector spaces $\mathbb{V}$ and $\mathbb{W}$ will be denoted $\Hom(\mathbb{V}, \mathbb{W})$ with $\Hom(\mathbb{V}, \mathbb{V})$ denoted specially as $\End(\mathbb{V})$. We write $\id$ for the identity map on any set.

For $X,Y \in \mathbb{R}^{n \times n}$, we write $[X,Y] = XY - YX$ for the commutator. We write $\mathfrak{so}(n) = \{X \in \mathbb{R}^{n \times n} : X^\tp = -X \}$ for the special orthogonal Lie algebra, i.e., the set of skew-symmetric matrices with $[ \, \cdot, \cdot \, ]$ as its Lie bracket.

We write $\mathcal{M}$ for a smooth manifold, $C^{\infty}(\mathcal{M})$ for its ring of smooth real-valued functions, $\mathbb{T}_x \mathcal{M} $ for its tangent space at $x \in \mathcal{M}$, and $\mathcal{X} (\mathcal{M})$ for its $C^\infty(\mathcal{M})$-module of smooth vector fields. If $\mathcal{M}$ is embedded in some ambient manifold, we write $ \mathbb{N}_x \mathcal{M}$ for its normal space at $x \in \mathcal{M}$.

We denote vector fields with an arrow like $\vec{v}$. The word \emph{tensor} in this article will always mean a tensor over a $C^\infty(\mathcal{M})$-module, i.e., a smooth tensor field on $\mathcal{M}$, and will be cast in the form of multilinear maps between tangent and normal spaces. With few exceptions, all multilinear maps in this article are defined at a specific point $x \in \mathcal{M}$. So as not to be overly verbose, we write ``on/of $\mathcal{M}$'' when we mean ``on/of $\mathcal{M}$ at $x$'' and we sometimes drop the subscript $x$ like in Table~\ref{tab:c} when there is no cause for confusion.

We write $\Gr(k,\mathbb{V})$ for the Grassmannian of $k$-dimensional subspaces in the vector space $\mathbb{V}$. We emphasize that in this article,
\begin{equation}\label{eq:inv}
\Gr(k,n) \coloneqq \{Q \in \mathbb{S}^n : Q^2 =I,\; \tr(Q)=2k - n\},
\end{equation}
i.e., $\Gr(k,n) \subseteq \mathbb{S}^n$ is the image of the embedding
\begin{equation}\label{eq:embed}
\varepsilon : \Gr(k,\mathbb{R}^n)  \to \mathbb{S}^n, \quad \mathbb{W} \mapsto P_{\mathbb{W}} - P_{\mathbb{W}^\perp} \eqqcolon Q_{\mathbb{W}},
\end{equation}
where $P_\mathbb{W} \in \mathbb{S}^n$ is the orthogonal projection matrix with image $\mathbb{W}$. So $\varepsilon$ sends a $k$-dimensional subspace $\mathbb{W} \subseteq \mathbb{R}^n$ to a matrix $Q_{\mathbb{W}} \in \mathbb{S}^n$. It is easy to verify \cite{ZLK20} that $Q_{\mathbb{W}}$ has the properties in \eqref{eq:inv}, $\varepsilon$ gives an embedding of Riemannian manifolds, and
\[
\Gr(k,n)  = \varepsilon\bigl(\Gr(k, \mathbb{R}^n) \bigr).
\]

\section{Curvature zoo}\label{sec:zoo}

We will review the definitions of various curvatures and curvature-related quantities. This section is not intended to be pedagogical, and only contains minimal commentaries. The goal is just to collect definitions scattered across standard references \cite{besse2007, do1992, KN, KMS93, Lee18, Petersen16} and some slightly less standard ones \cite{chen93,chen08,CK52,nonmet,Obata68} and present them in a unified set of notations (see Section~\ref{sec:note}) for the reader's easy reference.

All discussions below assume that $\mathcal{M}$ is a Riemannian manifold with Riemannian metric $\mathsf{g}$, i.e., $\mathsf{g}_x :  \mathbb{T}_x \mathcal{M} \times  \mathbb{T}_x \mathcal{M} \to \mathbb{R}$ defines an inner product at $x \in \mathcal{M}$. Section~\ref{sec:izoo} applies to $\mathcal{M}$ intrinsically. Section~\ref{sec:ezoo} applies when  $\mathcal{M}$ is embedded in an Euclidean space $\mathbb{E}^m$. More precisely by embedding we always mean an isometric embedding $\varepsilon : \mathcal{M} \to \mathbb{E}^m$ preserving the Riemannian metric, i.e., $\mathsf{g}_x(v,w) = \langle d_x\varepsilon(v), d_x\varepsilon(w) \rangle$ for all $x \in \mathcal{M}$, $v,w \in \mathbb{T}_x \mathcal{M}$. Henceforth, we will identify $\mathcal{M}$ with its image under the embedding $\varepsilon$ so that we have $\mathcal{M} \subseteq \mathbb{E}^m$ and  $\mathsf{g}_x(v,w) = \langle v, w \rangle$.

As usual, we will use the Levi-Civita connection throughout. We will say a few words about this choice in Section~\ref{sec:conn}.

\subsection{Extrinsic curvatures}\label{sec:ezoo}

In this section, $\mathcal{M} \subseteq \mathbb{E}^m$ is an $n$-dimensional submanifold of an $m$-dimensional Euclidean space. For any $x\in \mathcal{M}$,  we have the canonical identification
\[
\mathbb{E}^m \cong \mathbb{T}_x \mathbb{E}^m = \mathbb{T}_x \mathcal{M}\oplus \mathbb{N}_x \mathcal{M}
\]
and the corresponding orthogonal projections $\proj_{\mathbb{T}_x \mathcal{M}}$ and  $\proj_{\mathbb{N}_x \mathcal{M}}$.

The \emph{first fundamental form} is
\[
\fff_x: \mathbb{T}_x \mathcal{M} \times  \mathbb{T}_x \mathcal{M} \to \mathbb{R},\quad  \fff_x(v,w) \coloneqq \langle v,  w \rangle.
\]
This is nothing more than the Riemannian metric $\mathsf{g}$ on $\mathcal{M}$ expressed in terms of the inner product  $\langle \, \cdot, \cdot \, \rangle$ on $\mathbb{E}^m$.

The \emph{Gauss map} is defined by 
\[
\mathsf{\Gamma}: \mathcal{M} \to \Gr(m-n,  \mathbb{E}^m),\quad \mathsf{\Gamma}(x) \coloneqq \mathbb{N}_x \mathcal{M}.
\] 
Take any $\mathbb{V} \in \Gr(m-n, \mathbb{E}^m)$, note that this is an $(m-n)$-dimensional subspace of $\mathbb{E}^m$ and we have
\[
\mathbb{T}_{\mathbb{V}} \Gr(m-n, \mathbb{E}^m) \cong \Hom(\mathbb{V},  \mathbb{V}^{\perp}).
\]
In particular,  $\mathbb{T}_{\mathbb{N}_x \mathcal{M}} \Gr(m-n, \mathbb{E}^m) \cong \Hom(\mathbb{N}_x \mathcal{M},  \mathbb{T}_x \mathcal{M})$ and we may write the derivative of the Gauss map as
\begin{equation}\label{eq:identification}
d_x \mathsf{\Gamma}: \mathbb{T}_x \mathcal{M} \to \Hom(\mathbb{T}_x \mathcal{M},\mathbb{N}_x \mathcal{M}).
\end{equation}

The \emph{second fundamental form} $\sff_x$ of $\mathcal{M}$ is given by the derivative of the Gauss map in the form of \eqref{eq:identification},  i.e.,
\[
\sff_x: \mathbb{T}_x \mathcal{M} \times  \mathbb{T}_x \mathcal{M} \to \mathbb{N}_x \mathcal{M},\quad 
\sff_x(v,w) \coloneqq d_x \mathsf{\Gamma}(v)(w).
\]
This is likely the most important extrinsic differential geometric invariant of an embedded manifold. Indeed many of our curvatures, including intrinsic ones, will be derived from the second fundamental form. We define the \emph{index of relative nullity} \cite{CK52} of $\mathcal{M}$ as  
\[
\upnu_x \coloneqq \dim \{v \in \mathbb{T}_x \mathcal{M}: \sff_x(v,  w) = 0\text{~for all~}w\in \mathbb{T}_x \mathcal{M}\},
\]
noting that $\sff_x(v,w) = \sff_x(w,v)$. For each $\eta \in \mathbb{N}_x \mathcal{M}$,  we may regard $\langle \sff_x,  \eta \rangle$ as an endomorphism on $\mathbb{T}_x \mathcal{M}$ and this is called the \emph{Weingarten map} or \emph{shape operator},
\[
\mathsf{S}_x(\eta) : \mathbb{T}_x \mathcal{M} \to \mathbb{T}_x \mathcal{M}, \quad \langle \mathsf{S}_x(\eta)(v), w \rangle \coloneqq \langle \sff_x(v,w), \eta \rangle.
\]
This operator is self-adjoint as the second fundamental form is symmetric. The eigenvalues (necessarily real)
\[
\lambda_1(\eta),\dots,\lambda_n(\eta) \in \mathbb{R},
\]
are called the \emph{principal curvatures} of $\mathcal{M}$ along $\eta$. Their product is called the \emph{Gaussian curvature} of $\mathcal{M}$ along $\eta$
\[
\mathsf{G}_x(\eta) \coloneqq \det\mathsf{S}_x(\eta)  \in \mathbb{R},
\]
and their average is called the \emph{mean curvature} of $\mathcal{M}$ along $\eta$
\[ 
\mathsf{H}_x(\eta) \coloneqq \frac{1}{n}\tr\mathsf{S}_x(\eta)  \in \mathbb{R}.
\]
For any orthonormal basis $\eta_1,\dots,  \eta_{m-n} \in \mathbb{N}_x \mathcal{M}$, the \emph{mean curvature vector} of $\mathcal{M}$ is 
\[
\mathsf{H}_x \coloneqq \sum_{i = 1}^{m-n} \mathsf{H}_x(\eta_i) \eta_i  \in \mathbb{N}_x \mathcal{M}
\]
and its value is independent of the choice of orthonormal basis. Clearly, $\mathsf{H}_x (\eta) = \langle \mathsf{H}_x , \eta \rangle$.

The \emph{Gauss--Obata} map \cite{Obata68} of $\mathcal{M}$ is
\[
\mathsf{Q}_x: \mathbb{T}_x \mathcal{M} \to \mathbb{T}_x \mathcal{M},  \quad \mathsf{Q}_x(v) \coloneqq \sum_{j=1}^{m-n} \mathsf{S}_x(\eta_j)^2 (v).
\]
and the \emph{third fundamental form} is 
\[
\tff_x:  \mathbb{T}_x \mathcal{M} \times \mathbb{T}_x \mathcal{M} \to \mathbb{R}, \quad \tff_x(v,w) \coloneqq \langle \mathsf{Q}_x(v) ,  w \rangle.
\]
The version defined here differs from an alternative version defined in classical differential  \cite[Equation~21]{Chern} and algebraic \cite[Equations~1.45 and 1.46]{GH} geometry. The latter allows for $k$th fundamental forms for all $k \ge 4$. Nevertheless the important thing is that both versions agree with the classical third fundamental form for a surface $\mathcal{M} \subseteq \mathbb{R}^3$.

Getting slightly ahead of ourselves, the second fundamental form may also be expressed using the Levi-Civita connection $\nabla$ in \eqref{eq:levi}:
\begin{equation}\label{eq:2ndff}
\sff_x(v,w)  = \partial_{v} \vec{w} - \grad{v} \vec{w} = \proj_{\mathbb{N}_x \mathcal{M}}  (\partial_v \vec{w} )
\end{equation}
for any $v, w \in \mathbb{T}_x \mathcal{M}$ and where $\vec{w}\in \mathcal{X} (\mathcal{M})$ is any vector field with $\vec{w}(x) = w$ \cite[Theorem~8.2]{Lee18}.  The value of $\sff_x(v,w)$ is independent of the choice of $\vec{w}$ \cite[Proposition~8.1]{Lee18}.

\subsection{Intrinsic curvatures}\label{sec:izoo}

In this section $\mathcal{M}$ is a Riemannian manifold with metric tensor $\mathsf{g}$. One feature of our approach is that we will calculate intrinsic curvatures in extrinsic coordinates given by our involution model, vastly simplifying the work involved. As such it suffices to define the \emph{Levi-Civita connection} $\nabla$ of $\mathcal{M}$ as an embedded manifold:
\begin{equation}\label{eq:levi}
\nabla: \mathcal{X} (\mathcal{M}) \times \mathcal{X} (\mathcal{M}) \to \mathcal{X} (\mathcal{M}),\quad (\nabla(\vec{v},  \vec{w}))(x) \coloneqq \proj_{\mathbb{T}_x \mathcal{M}} ( \partial_v \vec{w}),
\end{equation}
where  $v \coloneqq \vec{v}(x) \in \mathbb{T}_x \mathcal{M}$ and $\partial_v \vec{w}$ is the standard directional derivative of the vector field, i.e., derivative of the vector-valued function $\vec{w} : \mathcal{M} \to \mathbb{E}^m$ along the direction $v \in \mathbb{T}_x \mathcal{M} \subseteq \mathbb{E}^m$. Note that this simple definition is possible only because both $\mathcal{M}$ and $\mathbb{T}_x \mathcal{M}$ are regarded as subsets of $\mathbb{E}^m$.  It is also common to write, for a fixed $v \in \mathbb{T}_x \mathcal{M}$,
\[
\grad{v} \vec{w} : \mathcal{X} (\mathcal{M}) \to \mathcal{X} (\mathcal{M}), \quad ( \grad{v} \vec{w} ) (x) \coloneqq  (\nabla(\vec{v},  \vec{w}))(x),
\]
as it behaves like a directional derivative.  A slight variation of this notation makes $v \in \mathbb{T}_x \mathcal{M}$ the variable and fixes $x \in \mathcal{M}$, giving
\[
\lgrad{x} \vec{w}: \mathbb{T}_x \mathcal{M} \to \mathbb{T}_x \mathcal{M}, \quad (\lgrad{x} \vec{w})(v) :=  (\grad{v} \vec{w})(x).
\]
Since $\lgrad{x} \vec{w}$ is a linear operator, it has a trace, which defines the \emph{divergence} for a vector field $\vec{w}$,
\[
\dive: \mathcal{X}(\mathcal{M}) \to C^\infty(\mathcal{M}),\quad \dive(\vec{w}) (x) \coloneqq \tr( \lgrad{x} \vec{w}).
\]
For higher-order tensor fields, the convention is to apply divergence to the last argument: If $\vec{w}_1,\dots,  \vec{w}_k \in \mathcal{X} (\mathcal{M})$, then
\[
\dive (\vec{w}_1 \otimes \cdots \otimes   \vec{w}_{k-1} \otimes \vec{w}_k ) \coloneqq  \dive (\vec{w}_k ) \vec{w}_1 \otimes \cdots \otimes \vec{w}_{k-1},
\]
and extended linearly to all $k$-tensor fields \cite{Petersen16}.

We will need two common notions \cite{besse2007} defined for any symmetric bilinear forms on $\mathbb{T}_x \mathcal{M}$. Let $\alpha,  \beta: \mathbb{T}_x \mathcal{M} \times \mathbb{T}_x \mathcal{M} \to \mathbb{R}$ be symmetric and bilinear,  their \emph{Kulkarni--Nomizu product} $\alpha \varowedge \beta$ is the symmetric quadrilinear form
\begin{equation}\label{eq:kn}
\begin{gathered}
\alpha \varowedge \beta : \mathbb{T}_x \mathcal{M} \times \mathbb{T}_x \mathcal{M}  \times \mathbb{T}_x \mathcal{M} \times \mathbb{T}_x \mathcal{M} \to \mathbb{R}, \\
\alpha \varowedge \beta(u,v,w,z) \coloneqq \alpha(u,w) \beta(v,z) -  \alpha(u,z) \beta(v,w) -  \alpha(v,w) \beta(u,z) + \alpha(v,z) \beta(u,w).
\end{gathered}
\end{equation}
If a symmetric bilinear form $\beta :\mathbb{T}_x \mathcal{M} \times \mathbb{T}_x \mathcal{M} \to \mathbb{R}$ satisfies
\[
(\grad{u} \beta)(v,w) = (\grad{v} \beta)(u,w) 
\]
for all $u, v, w \in \mathbb{T}_x \mathcal{M}$, then  it is called a \emph{Codazzi tensor}.

The \emph{Riemann curvature} or \emph{curvature tensor} of $\mathcal{M}$ is
\[
 \Rie_x :  \mathbb{T}_x\mathcal{M} \times  \mathbb{T}_x\mathcal{M} \times  \mathbb{T}_x\mathcal{M} \times \mathbb{T}_x\mathcal{M} \to \mathbb{R},\quad 
 \Rie_x (u,v,w,z) 
\coloneqq \langle \grad{u}  \grad{v} w -  \grad{v}  \grad{u} w -  \grad{[u, v]} w,  z \rangle.
\]
There is a common variant known by the same name:
\[
\mathsf{R}_x :  \mathbb{T}_x\mathcal{M} \times  \mathbb{T}_x\mathcal{M}  \to \End(\mathbb{T}_x \mathcal{M}), \quad \mathsf{R}_x(u,v)w \coloneqq \grad{u}  \grad{v} w -  \grad{v}  \grad{u} w -  \grad{[u, v]} w.
\]
Note that $\Rie_x(u,v,w,z) = \langle \mathsf{R}_x(u,v)w,  z \rangle$ with the slight difference being that $\mathsf{R}_x$ is a bilinear map and $\Rie_x$ is a quadrilinear form.  There is also a symmetric variant called the \emph{Jacobi tensor} of $\mathcal{M}$,
\[
\mathsf{J}_x:  \mathbb{T}_x\mathcal{M} \times  \mathbb{T}_x\mathcal{M} \times  \mathbb{T}_x\mathcal{M} \times \mathbb{T}_x\mathcal{M} \to \mathbb{R},\quad 
 \mathsf{J}_x (u,v,w,z) 
\coloneqq \frac{1}{2} \left( \Rie_x(u,v,w,z) + \Rie_x(w,v,u,z) \right).
\]
If $v, w \in \mathbb{T}_x \mathcal{M}$ are linearly independent,  then the \emph{sectional curvature} of $\mathcal{M}$ is 
\[
\upkappa_x : \mathbb{T}_x\mathcal{M} \times  \mathbb{T}_x\mathcal{M} \to  \mathbb{R},\quad \upkappa_x(v,w) \coloneqq \frac{\Rie_x(v,w, w,v)}{\lVert v \wedge w \rVert^2} = \frac{\Rie_x(v,w, w,v)}{\lVert v  \rVert^2 \lVert w \rVert^2 - \langle v,  w \rangle^2} .
\]
Let $v_1,\dots, v_n \in \mathbb{T}_x \mathcal{M}$ be an orthonormal basis.  The \emph{Ricci curvature} of $\mathcal{M}$ is 
\[
\Ric_x : \mathbb{T}_x\mathcal{M} \times \mathbb{T}_x\mathcal{M} \to \mathbb{R},\quad \Ric_x (v,w) \coloneqq \sum_{j=1}^n \Rie_x(v,  v_j,  v_j ,  w).
\]
The \emph{scalar curvature} of $\mathcal{M}$ is 
\[
\Sca_x \coloneqq \tr(\Ric_x) = \sum_{1 \le j < k \le n} \upkappa_x(v_j,  v_k) =
 \sum_{1 \le j < k \le n} \Rie_x(v_j,  v_k,  v_k ,  v_j).
\]
The last two curvatures gives us the \emph{traceless Ricci curvature},
\[
\mathsf{Z}_x : \mathbb{T}_x\mathcal{M} \times \mathbb{T}_x\mathcal{M} \to \mathbb{R},\quad \mathsf{Z}_x (v,w) \coloneqq  \Ric_x (v,w) - \frac{\Sca_x}{n} \mathsf{g}_x(v,w),
\]
important as it gives a $\mathsf{g}$-orthogonal decomposition of Ricci curvature. A manifold with $\mathsf{Z} = 0$ is called an Einstein manifold \cite{besse2007}.

The scalar curvature allows a generalization to any $d$-dimensional subspace $\mathbb{V} \subseteq \mathbb{T}_x \mathcal{M}$,
\[
\Sca_x(\mathbb{V}) \coloneqq \sum_{1 \le j < k \le d} \upkappa_x(v_j,  v_k),
\]
where $v_1,\dots,  v_d \in \mathbb{V}$ is any orthonormal basis. Clearly $\Sca_x(\mathbb{T}_x \mathcal{M}) = \Sca_x$. From these one may construct the upper and lower \emph{delta invariants} \cite{chen93,chen08}, given respectively by
\begin{align*}
\overline{\updelta}_x (d_1,\dots,d_r) &\coloneqq \Sca_x - \inf_{\substack{ \dim \mathbb{V}_j = d_j \\ \mathbb{V}_j \perp \mathbb{V}_k, \; j< k}} \biggl[ \sum_{j=1}^r \Sca_x(\mathbb{V}_j) \biggr],  \\
\underline{\updelta}_x (d_1,\dots,d_r) &\coloneqq \Sca_x - \sup_{\substack{ \dim \mathbb{V}_j = d_j \\ \mathbb{V}_j \perp \mathbb{V}_k, \; j< k}} \biggl[ \sum_{j=1}^r \Sca_x(\mathbb{V}_j) \biggr],
\end{align*}
where $d_1,\dots,  d_r \in \mathbb{Z}$ are such that  $2 \le d_1 \le \cdots \le d_r$ and  $d_1 + \dots + d_r \le \dim \mathcal{M}$.

We will next define a well-known quartet of closely related curvature tensors \cite{besse2007}. The \emph{Schouten tensor} of $\mathcal{M}$ is 
\[
\mathsf{P}_x : \mathbb{T}_x \mathcal{M} \times \mathbb{T}_x \mathcal{M} \to \mathbb{R},\quad \mathsf{P}_x(v,w) \coloneqq \frac{1}{n-2} \Bigl( \Ric_x(v,w) - \frac{\Sca_x}{2(n-1)} \mathsf{g}_x(v, w) \Bigr). 
\]  
The \emph{Cotton tensor} of $\mathcal{M}$ is
\[
\mathsf{C}_x : \mathbb{T}_x \mathcal{M} \times \mathbb{T}_x \mathcal{M} \times \mathbb{T}_x \mathcal{M} \to \mathbb{R}, \quad  \mathsf{C}_x(u,v,w) \coloneqq (\grad{u}  \mathsf{P})_x (v,w) - (\grad{v} \mathsf{P})_x (u,w).
\]
The \emph{Weyl tensor} of $\mathcal{M}$ is 
\[
\mathsf{W}_x  : \mathbb{T}_x \mathcal{M} \times \mathbb{T}_x \mathcal{M} \times \mathbb{T}_x \mathcal{M} \times \mathbb{T}_x \mathcal{M} \to \mathbb{R},\quad
\mathsf{W}_x  \coloneqq   \Rie_x- \frac{1}{n-2} \mathsf{Z}_x \varowedge \mathsf{g}_x - \frac{\Sca_x}{2n(n-1)} \mathsf{g}_x \varowedge \mathsf{g}_x.
\]
The \emph{Bach} tensor of $\mathcal{M}$ is
\begin{gather*}
\mathsf{B}_x : \mathbb{T}_x \mathcal{M} \times  \mathbb{T}_x \mathcal{M} \to \mathbb{R}, \\ 
\mathsf{B}_x(u,w) \coloneqq \frac{1}{n-3} \sum_{i,j=1}^n (\nabla^2_{v_i,  v_j} \mathsf{W})_x(u, v_i, v_j, w) + \frac{1}{n-2} \sum_{i,j=1}^n \Ric_x(v_i, v_j) \mathsf{W}_x(u, v_i, v_j,  w),
\end{gather*}

We will also describe some intrinsic curvatures that are more typically studied in non-Riemannian geometry, i.e., for a connection $\nabla$ other than the Levi-Civita connection (see Section~\ref{sec:conn}).

The \emph{torsion tensor} \cite[Volume~I, Section~III.5]{KN} of $\mathcal{M}$ is
\[
\mathsf{T}_x: \mathbb{T}_x \mathcal{M} \times \mathbb{T}_x \mathcal{M}  \to \mathbb{R}, \quad
\mathsf{T}_x(v,w) \coloneqq \grad{v} w - \grad{w} v - [v,w].
\]
The \emph{nonmetricity tensor} \cite{nonmet} of $\mathcal{M}$ is
\[
\mathsf{Q}_x: \mathbb{T}_x \mathcal{M} \times \mathbb{T}_x \mathcal{M} \times \mathbb{T}_x \mathcal{M} \to \mathbb{R}, \quad
\mathsf{Q}_x(u,v,w) \coloneqq -\grad{u} \mathsf{g}_x(v, w) + \mathsf{g}_x(\grad{u} v, w) + \mathsf{g}_x(v, \grad{u} w).
\]
For any vector subbundle $\mathbb{V}\mathcal{M}$ of $\mathbb{T} \mathcal{M}$ with projection $\pi : \mathbb{T} \mathcal{M} \to \mathbb{V}\mathcal{M}$, its \emph{cocurvature} \cite{Michor88, KMS93} is
\[
\mathsf{R}^*_\pi :\mathcal{X} (\mathcal{M}) \times \mathcal{X} (\mathcal{M}) \to \mathcal{X} (\mathcal{M}), \quad
\mathsf{R}^*_\pi(\vec{v}, \vec{w}) \coloneqq (\id - \pi)\bigl( [\pi(\vec{v}), \pi(\vec{w})] \bigr).
\] 
In general $\mathsf{R}^*_\pi$ measures the failure of integrability of $\mathbb{V}\mathcal{M}$.  To put the cocurvature in perspective, the \emph{curvature} in this context is
\[
\mathsf{R}_\pi: \mathcal{X} (\mathcal{M}) \times \mathcal{X} (\mathcal{M}) \to \mathcal{X} (\mathcal{M}), \quad
\mathsf{R}_\pi (\vec{v}, \vec{w}) \coloneqq \pi \bigl([(\id- \pi)(\vec{v}), (\id - \pi)(\vec{w})]\bigr).
\]
Evidently, this is a more general notion than the Riemann curvature $\mathsf{R}$ but we will see how they are related in Proposition~\ref{prop:zero}.

\subsection{Non-Riemannian curvatures}\label{sec:conn}

Our choice of Levi-Civita connection is all but preordained by the Fundamental Theorem of Riemannian Geometry \cite[Chapter~2, Theorem~3.6]{do1992}, namely, it is the unique affine connection that is torsion-free and metric-compatible. The following proposition is stated for completeness, essentially a reminder that  non-Riemannian curvatures like torsion, nonmetricity, and cocurvature are trivial for a Riemannian manifold.
\begin{proposition}\label{prop:zero}
Let $\nabla$ be the Levi-Civita connection on $\mathcal{M}$.
\begin{enumerate}[\upshape (a)]
\item The torsion tensor and the nonmetricity tensor vanish identically, i.e.,
\[
\mathsf{T}_x(v,w) = 0, \qquad \mathsf{Q}_x(u,v,w) = 0
\]
for all $x \in \mathcal{M}$ and all $u,v,w \in \mathbb{T}_x \mathcal{M}$.

\item Let $\pi: \O(\mathcal{M})\to \mathcal{M}$ be the orthonormal frame bundle on $\mathcal{M}$, $\mathbb{V} \mathcal{M} \coloneqq \ker (d \pi)$, and $\widehat{\pi}: \mathbb{T}\O(\mathcal{M}) \to \mathbb{V} \mathcal{M}$ the projection induced by $\nabla$. Then the cocurvature vanishes and the curvature equals the Riemann curvature up to sign, i.e.,
\[
\mathsf{R}^*_{\widehat{\pi}}(\vec{v}, \vec{w})  = 0, \qquad \mathsf{R}_{\widehat{\pi}}(\vec{v}, \vec{w}) = - \mathsf{R}(\vec{v}, \vec{w})
\]
where the first equality holds for all $\vec{v}, \vec{w} \in \mathcal{X} (\O(\mathcal{M}))$, the second for all $\vec{v}, \vec{w} \in \mathcal{X} ( \mathcal{M} ) \subseteq \mathcal{X}(\O(\mathcal{M}))$.
\end{enumerate}
\end{proposition}
\begin{proof}
Since the Levi-Civita connection on $\mathcal{M}$ is, by definition, the unique connection that is torsion free and compatible with the metric $\mathsf{g}$, both $\mathsf{T}$ and $\mathsf{Q}$ vanish.

The horizontal bundle $\ker \widehat{\pi}$ is $\mathbb{T} \mathcal{M}$,  whose integral manifold is $\mathcal{M}$. So the cocurvature $\mathsf{R}^*_{\widehat{\pi}}$ must vanish identically.  Since $\O(\mathcal{M})$ consists of orthonormal frames on  $\mathcal{M}$,  the typical fiber of $\mathbb{V} (\mathcal{M}) = \ker (d \pi)$ is $\mathfrak{so}(n)$,  where $n = \dim \mathcal{M}$. As $\widehat{\pi}$ is induced by $\nabla$, it can be regarded as an $\mathfrak{so}(n)$-valued differential $1$-form $\omega$ on $\O(\mathcal{M})$ \cite[Chapter~1]{Bleecker}. The projection map $\widehat{\pi}$ gives a decomposition
\[
\mathbb{T} \O(\mathcal{M}) \simeq \mathbb{V} \mathcal{M} \oplus \mathbb{T} \mathcal{M}.
\] 
Thus for any $\vec{v}, \vec{w} \in \mathcal{X} (\O(\mathcal{M}))$,  we may identify $(\id - \widehat{\pi})(\vec{v})$ and $(\id - \widehat{\pi})(\vec{w})$ with elements in $\mathcal{X}(\mathcal{M})$ and  
\[
-\mathsf{R}_{\widehat{\pi}}(\vec{v},  \vec{w}) = -\omega ([(\id - \widehat{\pi}) (\vec{v}),  (\id - \widehat{\pi}) (\vec{w})]) = d \omega \bigl( (\id - \widehat{\pi}) (\vec{v}),  (\id - \widehat{\pi}) (\vec{w}) \bigr).
\]
In other words, $-\mathsf{R}_{\widehat{\pi}}$ turns out to be the curvature $2$-form of $\omega$, which by \cite[Volume~I, Section~III.5]{KN} is equal to the Riemann curvature $\mathsf{R}$.
\end{proof}

\section{Extrinsic curvatures of the Grassmannian}\label{sec:extrinsic}

We are now in a position to calculate various curvatures of the Grassmannian modeled as $\Gr(k,n)$ and express them as simple matrix formulas. Our strategy is first to calculate the extrinsic curvatures in this section, notably the second fundamental form, and then use it as the basis for our calculation of intrinsic curvatures in Section~\ref{sec:intrinsic}.

Our ambient Euclidean space of choice is  $\mathbb{S}^n$ equipped with the standard (also called trace or Frobenius) inner product on $\mathbb{R}^{n\times n}$ given by
\[
\langle X, Y \rangle \coloneqq \tr(X^\tp Y) = \sum_{i=1}^n \sum_{j=1}^n x_{ij} y_{ij}.
\]
When restricted to $\Gr(k,n)$, it gives us our Riemannian metric
\begin{equation}\label{eq:metric}
\mathsf{g}_Q : \mathbb{T}_{Q} \Gr(k,n)  \times \mathbb{T}_{Q} \Gr(k,n)  \to \mathbb{R}, \quad \mathsf{g}_Q(X,Y) = \tr(X^\tp Y)
\end{equation}
for all $Q \in \Gr(k,n)$. Of course, we have $X^\tp = X$ as $X \in \mathbb{S}^n$ but we choose to keep the transpose in our notation to remind ourselves that this is the trace inner product.

Given $Q\in \Gr(k,n) \subseteq \mathbb{S}^n$, we have an eigenvalue decomposition $Q = V I_{k,n-k} V^\tp$ for some $V\in \O(n)$ and $I_{k,n-k} \coloneqq \diag(I_k,-I_{n-k}) = \diag(1,\dots,1,-1,\dots,-1)$. The tangent and normal spaces of $\Gr(k,n)$ at $Q$ are
\begin{align}
\mathbb{T}_{Q} \Gr(k,n) &= \biggl\lbrace
V \begin{bmatrix}
0 & X_0 \\
X_0^\tp & 0 
\end{bmatrix} V^\tp \in \mathbb{S}^n : X_0 \in \mathbb{R}^{k\times (n-k)}
\biggr\rbrace,  \label{eq:tangent} \\
\mathbb{N}_Q \Gr(k,n) &= \biggl\lbrace
V \begin{bmatrix}
H_1 & 0 \\
0 & H_2 
\end{bmatrix} V^\tp \in \mathbb{S}^n  : H_1 \in \mathbb{S}^k, \; H_2 \in \mathbb{S}^{n-k}
\biggr\rbrace.  \label{eq:normal} 
\end{align}
Henceforth we will consistently write any point $Q\in \Gr(k,n)$, tangent vector $X \in \mathbb{T}_Q \Gr(k,n)$ and normal vector $H \in \mathbb{N}_Q \Gr(k,n)$ as
\begin{equation}\label{eq:para}
Q = V \begin{bmatrix}
I_k & 0\\
0 & -I_{n-k}
\end{bmatrix} V^\tp ,\quad
X = V \begin{bmatrix}
0 & X_0 \\
X_0^\tp & 0
\end{bmatrix} V^\tp ,\quad
H = V \begin{bmatrix}
H_1 & 0 \\
0 & H_2
\end{bmatrix} V^\tp.
\end{equation} 
These three decompositions in \eqref{eq:para} are easily computable: One may first compute $V$ with a QR decomposition of $(I + Q)/2$ \cite[Lemma~7.1]{ZLK20}.  Thereafter, $X_0$ and $(H_1,H_2)$ can be readily determined by computing $V^\tp X V$ and $V^\tp H V$ respectively. The ease of computing $V$, $X_0$, and $(H_1, H_2)$ translates directly into the ease of computing the curvatures in Table~\ref{tab:c}.

The simple parametrization of these three basic objects in the involution model is a key to the simplicity of our calculations.
A convenient orthonormal basis of $\mathbb{T}_{I_{k,n-k}} \Gr(k,n)$ is given by
\begin{equation}\label{eq:onb}
\biggl\lbrace \dfrac{\sqrt{2}}{2}\begin{bmatrix} 
0 & E_{ij} \\
E_{ij}^\tp & 0 
\end{bmatrix} \in \mathbb{S}^n  : i =1,\dots, k,\;  j = 1,\dots, n-k \biggr\rbrace
\end{equation}
where $E_{ij}$ is the $k\times (n-k)$ matrix with one in the $(i,j)$th entry and zero everywhere else. We refer readers to \cite[Section~3]{ZLK20} for the proofs of statements in this and the last paragraph. There is an important distinction between our approach in \cite{ZLK20} and that in this paper, namely, the ambient manifold is $\O(n)$ in \cite[Section~3]{ZLK20} whereas here it is $\mathbb{S}^n$. For instance, this accounts for the difference between our expression for normal space in \eqref{eq:normal} and that in  \cite[Proposition~3.6]{ZLK20}.

The next two results require no calculation and are just stated for completeness. 
\begin{proposition}[First fundamental form]\label{prop:fff}
The first fundamental form $\fff_Q: \mathbb{T}_Q \Gr(k,n) \times \mathbb{T}_Q\Gr(k,n) \to  \mathbb{R}$ is given by
\[
\fff_Q(X, Y) = 2  \tr(X_0^\tp Y_0)
\]
with $Q$, $X$, $Y$ parameterized as in \eqref{eq:para}.
\end{proposition}
\begin{proposition}[Gauss map]\label{prop:Gmap}
The Gauss map of $\Gr(k,n)$ in $\mathbb{S}^n$ is given by 
\begin{align*}
\mathsf{\Gamma} : \Gr(k,n) &\to \Gr\bigl( \tbinom{n+1}{2} - k(n-k),  \mathbb{S}^n \bigr), \\
\mathsf{\Gamma} (Q) &= \mathbb{N}_Q (\Gr(k,n)) =V  \left\lbrace
\begin{bmatrix}
H_1 & 0 \\
0 & H_2 
\end{bmatrix} : H_1 \in \mathbb{S}^k, H_2 \in \mathbb{S}^{n-k}
\right\rbrace V^\tp,
\end{align*}
with $Q$ parameterized as in \eqref{eq:para}.
\end{proposition}

The next calculation is our key to unlocking other calculations in this article.
\begin{theorem}[Second fundamental form]\label{thm:2ndFF}
The second fundamental form $\sff_Q: \mathbb{T}_Q \Gr(k,n) \times \mathbb{T}_Q\Gr(k,n) \to \mathbb{N}_Q \Gr(k,n)$ is given by
\[
\sff_Q(X,Y) =\frac{1}{2} V \begin{bmatrix}
-X_0 Y_0^\tp - Y_0 X_0^\tp & 0 \\
0 & X_0^\tp Y_0 + Y_0^\tp X_0
\end{bmatrix} V^\tp,
\]
with $Q$, $X$, $Y$ parameterized as in \eqref{eq:para}.
\end{theorem}
\begin{proof}
Since $\O(n)$ acts on $\Gr(k,n)$ transitively and isometrically, it suffices to calculate $\sff$ at $I_{k,n-k}\in \Gr(k,n)$. In this case $X, Y \in \mathbb{T}_{I_{k,n-k}} \Gr(k,n)$ may be written as
\[
X = \begin{bmatrix}
0 & X_0 \\
X_0^\tp & 0
\end{bmatrix},\quad Y = \begin{bmatrix}
0 & Y_0 \\
Y_0^\tp & 0
\end{bmatrix}
\]
for some $X_0 ,Y_0 \in \mathbb{R}^{k \times (n-k)}$. Points near $I_{k,n-k}$ in $\mathbb{S}^k$ can be parametrized as 
\[
\varphi(B, H_1, H_2) = \exp \biggl( \frac{1}{2} \begin{bmatrix} 
0 & -B \\
B^\tp & 0
\end{bmatrix} \biggr) \biggl(I_{k,n-k} + \begin{bmatrix}
H_1 & 0 \\
0 & H_2 
\end{bmatrix} \biggr)
\exp \biggl(   \frac{1}{2} \begin{bmatrix}
0 & -B \\
B^\tp & 0
\end{bmatrix}^\tp \biggr),
\]
where $B\in \mathbb{R}^{k\times (n-k)}$, $H_1 \in \mathbb{S}^k$, and $H_2 \in \mathbb{S}^{n-k}$ have sufficiently small norms. This follows from the implicit function theorem,  as the Jacobian of $\varphi$ at $(0, 0, 0)$ gives an isometry between $\mathbb{R}^{k\times (n-k)} \oplus \mathbb{S}^k \oplus \mathbb{S}^{n-k}$ and $\mathbb{T}_{I_{k, n-k}} \Gr(k,n) \oplus \mathbb{N}_{I_{k, n-k}} \Gr(k,n)$.

Clearly, we have $\varphi(B, H_1, H_2) \in \Gr(k,n)$ if and only if $H_1 = 0$ and $H_2 = 0$. Thus we may extend $X$ by 
\[
\widetilde{X}\bigl(\varphi(B, H_1, H_2) \bigr) = 
\exp \biggl(  \frac{1}{2} \begin{bmatrix}
0 & -B \\
B^\tp & 0
\end{bmatrix} \biggr) 
\begin{bmatrix}
H_1 & X_0 \\
X_0^\tp & H_2 
\end{bmatrix} 
\exp \biggl(  \frac{1}{2} \begin{bmatrix}
0 & -B \\
B^\tp & 0
\end{bmatrix} \biggr)^\tp.
\]
Such an $\widetilde{X}$ is an extension of a local vector field around $I_{k,n-k}$ on $\Gr(k,n)$. By \eqref{eq:2ndff}, 
\[
\sff_{I_{k,n-k}}(X,Y) =\proj_{\mathbb{N}_{I_{k,n-k}} \Gr(k,n)} \bigl( \langle \widetilde{\nabla} \widetilde{X}(I_{k,n-k}), Y \rangle \bigr)
\]
where $\widetilde{\nabla}$ denotes the covariant derivative in the Euclidean space $\mathbb{S}^n$, i.e.,
\[
\widetilde{\nabla} \widetilde{X} = \bigl( 
\partial_{B} \widetilde{X}, 
\partial_{H_1} \widetilde{X}, 
\partial_{H_2} \widetilde{X}
\bigr)
\]
Since $Y$ is a tangent vector, we obtain 
\[
\bigl\langle \widetilde{\nabla} \widetilde{X}(I_{k,n-k}), Y \bigr\rangle =\sum_{i=1}^k \sum_{j=1}^{n-k} \frac{\partial \widetilde{X}}{\partial b_{ij}}(I_{k,n-k}) y_{0ij},
\]
where we have written $B = (b_{ij})$ and $Y_0 = (y_{0ij})$. 
Observe that 
\begin{align*}
\frac{\partial \widetilde{X}}{\partial b_{ij}}(I_{k,n-k}) &= 
-\frac{1}{2}\begin{bmatrix}
 0 & E_{ij} \\
 -E_{ij}^\tp & 0
 \end{bmatrix}\begin{bmatrix}
0 & X_0 \\
X_0^\tp & 0
\end{bmatrix}  +  \frac{1}{2}\begin{bmatrix}
0 & X_0 \\
X_0^\tp & 0
\end{bmatrix}  \begin{bmatrix}
 0 & E_{ij} \\
 -E_{ij}^\tp & 0
 \end{bmatrix} \\
&=   
\frac{1}{2} \begin{bmatrix}
-E_{ij} X_0^\tp - X_0 E_{ij}^\tp & 0 \\
0 & E_{ij}^\tp X_0 + X_0^\tp E_{ij} 
\end{bmatrix}.
\end{align*}
Therefore we have 
\begin{align*}
\sum_{i=1}^k \sum_{j=1}^{n-k} \frac{\partial \widetilde{X}}{\partial b_{ij}}(I_{k,n-k}) y_{0ij} 
&=\frac{1}{2} \sum_{i=1}^k \sum_{j=1}^{n-k}  \begin{bmatrix}
 -(E_{ij} X_0^\tp + X_0 E_{ij}^\tp) y_{0ij} & 0 \\
0 &  (E_{ij}^\tp X_0 + X_0^\tp E_{ij}) y_{0ij} 
\end{bmatrix} \\
&= \frac{1}{2} \begin{bmatrix}
-X_0 Y_0^\tp - Y_0 X_0^\tp & 0 \\
0 & X_0^\tp Y_0 + Y_0^\tp X_0
\end{bmatrix},
\end{align*}
where the last expression is our required $\sff_{I_{k,n-k}}(X,Y) $.
\end{proof}

We record an observation that follows from an additional step of singular value decomposition.
\begin{corollary}[Index of relative nullity]\label{cor:nullity}
The index of relative nullity $\upnu_Q $ of $\Gr(k,n)$ is zero. 
\end{corollary}
\begin{proof}
Let $X\in \mathbb{T}_Q \Gr(k,n)$ be such that $\sff_Q(X, Y) = 0$ for all $Y\in \mathbb{T}_Q \Gr(k,n)$, with $Q,  X,Y$ parametrized as in  \eqref{eq:para}. We claim that $X = 0$. By Theorem~\ref{thm:2ndFF}, we must have
\begin{equation}\label{cor:nullity:eq1}
X_0 Y_0^\tp + Y_0 X_0^\tp = 0,\quad X_0^\tp Y_0 + Y_0^\tp X_0 = 0
\end{equation}
for any $Y_0 \in \mathbb{R}^{k \times (n-k)}$. Let $X_0 = U \Sigma V^\tp$ be a singular value decomposition with $U\in \O(k)$ and $V\in \O(n-k)$. Then \eqref{cor:nullity:eq1} becomes
\[
\Sigma  (U^\tp Y_0 V)^\tp + (U^\tp Y_0 V) \Sigma = 0,\quad \Sigma (U^\tp Y_0 V) + (U^\tp Y_0 V)^\tp \Sigma = 0.
\]
Since $Y_0$ is arbitrary, we may set $X_0 = \Sigma$ in \eqref{cor:nullity:eq1}. Now by taking $Y_0$ to be an arbitrary diagonal $k\times (n-k)$ matrix, we see that $\Sigma = 0$. Hence $X_0 = 0$ and $X = 0$.
\end{proof}

The Weingarten map is an alternative way to express the second fundamental form and thus follows easily from Theorem~\ref{thm:2ndFF}.
\begin{corollary}[Weingarten map]\label{cor:wein}
The Weingarten map $\mathsf{S}_Q(H): \mathbb{T}_Q \Gr(k,n) \to \mathbb{T}_Q \Gr(k,n)$ along the normal direction $H \in\mathbb{N}_Q \Gr(k,n)$ is given by
\begin{equation}\label{eq:wein}
\mathsf{S}_Q(H)(X) = \frac12 V
\begin{bmatrix}
0 &  X_0 H_2 - H_1 X_0  \\
( X_0 H_2 - H_1 X_0 )^\tp & 0
\end{bmatrix} V^\tp
\end{equation}
with $Q$, $X$, $H$ parameterized as in \eqref{eq:para}.
\end{corollary}
\begin{proof}
We plug in the expressions from \eqref{eq:para} into $\langle \mathsf{S}_Q(H)(X),Y \rangle =  \langle \sff_Q(X,Y) , H \rangle $ and use standard properties of trace to get
\begin{align*}
\langle \mathsf{S}_Q(H)(X),Y \rangle 
&= \frac12 \bigl[ \tr\bigl( -(X_0 Y_0^\tp + Y_0 X_0^\tp ) H_1 \bigr) + \tr\bigl( (X_0^\tp  Y_0 + Y_0^\tp  X_0) H_2 \bigr) \bigr] \\
&= \frac12 \bigl[ \tr \bigl( ( X_0 H_2 - H_1X_0) \bigr) Y_0^\tp )+ \tr\bigl( ( H_2 X_0^\tp - X_0^\tp  H_1)  Y_0 \bigr) \bigr] \\
&= \tr\biggl( \frac12
\begin{bmatrix}
0 &  X_0 H_2 - H_1 X_0  \\
( X_0 H_2 - H_1 X_0 )^\tp & 0
\end{bmatrix}
\begin{bmatrix}
0 &  Y_0  \\
Y_0^\tp & 0
\end{bmatrix}
\biggr),
\end{align*}
and thereby deducing \eqref{eq:wein}.
\end{proof}

The calculation of mean curvature is also straightforward.
\begin{corollary}[Mean curvature]\label{cor:meancurv}
The mean curvature vector of $\Gr(k,n)$ is given by
\[
\mathsf{H}_Q = \frac{1}{k(n-k)} \tr(\sff_Q) = \frac{1}{2k(n-k)} V\begin{bmatrix}
-(n-k) I_k & 0 \\
0 & k I_{n-k}
\end{bmatrix} V^\tp
\]
and the mean curvature of $\Gr(k,n)$ along $H  \in \mathbb{N}_Q \Gr(k,n)$ is given by
\[
\mathsf{H}_Q (H) =\frac{ (k-n)\tr H_1 + k \tr H_2 }{2k(n-k)}
\]
with $Q$, $X$, $H$ parameterized as in \eqref{eq:para}.
\end{corollary}
\begin{proof}
We use the orthonormal basis of $\mathbb{T}_{I_{k,n-k}} \Gr(k,n)$ in \eqref{eq:onb}. A straightforward but slightly messy calculation gives
\begin{align*}
(E_{ij} E_{i'j'}^\tp)_{pq} &= \begin{cases}
1 &\text{if } j = j' \text{ and } (p,q) = (i,i'), \\
0 &\text{otherwise},
\end{cases}
\shortintertext{for any $p,q \in \{1,\dots,k\}$; and}
(E_{ij}^\tp E_{i'j'})_{pq} &= \begin{cases}
1 &\text{if } i = i' \text{ and } (p,q) = (j,j') \\
0 &\text{otherwise}
\end{cases}
\end{align*}
for any $p,q \in \{1,\dots,n-k\}$.  Using these, we may evaluate
\begin{align*}
\sff_{I_{k,n-k}}
\biggl(\frac{\sqrt{2}}{2} \begin{bmatrix} 
0 & E_{ij} \\
E_{ij}^\tp & 0 
\end{bmatrix}
,  \frac{\sqrt{2}}{2}\begin{bmatrix} 
0 & E_{i'j'} \\
E_{i'j'}^\tp & 0 
\end{bmatrix}
\biggr)
&=\frac{1}{4} \begin{bmatrix}
-\delta_{jj'} (E_{ii'} + E_{i'i}) & 0 \\
0 & \delta_{ii'} (E_{jj'} + E_{j'j}) 
\end{bmatrix} \\
& =\frac{\delta_{jj'} \delta_{ii'}}{2}  \begin{bmatrix}
- E_{ii} & 0 \\
0 & E_{jj}
\end{bmatrix} 
\end{align*}
and obtain the required expression by summing over the basis. The mean curvature along $H$ is then calculated from $\mathsf{H}_Q (H) = \langle \mathsf{H}_Q,  H \rangle$.
\end{proof}

\begin{corollary}[Principal and Gaussian curvatures]\label{cor:prin}
Let $Q \in \Gr(k,n)$ and $H \in \mathbb{N}_Q \Gr(k,n)$ be parameterized as in \eqref{eq:para}. Then the Weingarten map $\mathsf{S}_Q(H)$ has eigenpairs given by
\[
\left( \frac12 (\lambda_{k+j} - \lambda_i), V \begin{bmatrix}
0 & Q_1 E_{ij} Q_2^\tp \\
Q_2 E_{ij}^\tp Q_1^\tp & 0 
\end{bmatrix} V^\tp \right),\quad  i=1,\dots,k, \; j =1,\dots, n-k,
\]
where $H_1 = Q_1 \Lambda_1 Q_1^\tp $ and $H_2 = Q_2 \Lambda_2 Q_2^\tp $ are eigenvalue decompositions with $\Lambda_1 = \diag(\lambda_1,\dots, \lambda_k)$ and $\Lambda_2 = \diag(\lambda_{k+1},\dots, \lambda_n)$.
\begin{enumerate}[\upshape (a)]
\item The principal curvatures of $\Gr(k,n)$ along $H$ are 
\[
\upkappa_{ij} = \frac12 (\lambda_{k+j} - \lambda_i),\quad  i=1,\dots,k, \; j =1,\dots, n-k.
\]

\item The Gaussian curvature of $\Gr(k,n)$ along $H$ is 
\[
\mathsf{G}_Q(H) = \frac{1}{2^{k(n-k)}}\prod_{i=1}^k \prod_{j=1}^{n-k} (\lambda_{k+j} - \lambda_i).
\]
\end{enumerate}
\end{corollary}
\begin{proof}
By \eqref{eq:wein}, we have
\[
\mathsf{S}_Q(H)\left(
V \begin{bmatrix}
0 & X_0 \\
X_0^\tp & 0
\end{bmatrix} V^\tp\right) = 
V \mathsf{S}_{I_{k,n-k}}\left( H_0  \right)
\left(
\begin{bmatrix}
0 & X_0 \\
X_0^\tp & 0
\end{bmatrix}
\right) V^\tp, \quad H_0 \coloneqq \begin{bmatrix}
H_1 & 0 \\
0 & H_2
\end{bmatrix}.
\]
Write $\Lambda \coloneqq \diag (\Lambda_1,\Lambda_2) = \diag(\lambda_1,\dots, \lambda_n)$. Then
\begin{align*}
\mathsf{S}_{I_{k,n-k}}(H_0)
\left( \begin{bmatrix}
0 & X_0 \\
X_0^\tp & 0
\end{bmatrix} \right) 
&= \frac{1}{2} \begin{bmatrix}
0 & H_1 X_0 - X_0 H_2 \\
(H_1 X_0 - X_0 H_2)^\tp & 0
\end{bmatrix} \\
&= \frac{1}{2} \begin{bmatrix}
0 & Q_1 \Lambda_1 Q_1^\tp X_0 - X_0 Q_2 \Lambda_2 Q_2^\tp \\
(Q_1 \Lambda_1 Q_1^\tp X_0 - X_0 Q_2 \Lambda_2 Q_2^\tp)^\tp & 0
\end{bmatrix} \\
&= \frac{1}{2} \begin{bmatrix}
Q_1 & 0 \\
0 & Q_2 
\end{bmatrix}
\begin{bmatrix}
0 &  Y_0  \\
Y_0^\tp & 0 
\end{bmatrix}
\begin{bmatrix}
Q_1 & 0 \\
0 & Q_2 
\end{bmatrix}^\tp \\
&=\begin{bmatrix}
Q_1 & 0 \\
0 & Q_2 
\end{bmatrix} \mathsf{S}_{I_{k,n-k}}(\Lambda) \biggl(
\begin{bmatrix}
0 & Q_1^\tp X_0 Q_2 \\ 
(Q_1^\tp X_0 Q_2)^\tp  & 0 
\end{bmatrix}
\biggr) \begin{bmatrix}
Q_1 & 0 \\
0 & Q_2 
\end{bmatrix}^\tp,
\end{align*}
where $Y_0 \coloneqq \Lambda_1 (Q_1^\tp X_0 Q_2) - (Q_1^\tp X_0 Q_2) \Lambda_2$. So it suffices to diagonalize the linear operator $\mathsf{S}_{I_{k,n-k}} (\Lambda) : \mathbb{T}_{I_{k,n-k}} \Gr(k,n) \to \mathbb{T}_{I_{k,n-k}} \Gr(k,n)$. Now observe that
\[
\mathsf{S}_{I_{k,n-k}} (\Lambda) \biggl( \frac{\sqrt{2}}{2} \begin{bmatrix} 
0 & E_{ij} \\
E_{ij}^\tp & 0 
\end{bmatrix} \biggr)  
= \frac{ \lambda_{k + j} - \lambda_i}{2} \biggl( \frac{\sqrt{2}}{2} \begin{bmatrix} 
0 & E_{ij} \\
E_{ij}^\tp & 0 
\end{bmatrix} \biggr)
\]
for $i=1,\dots,k$ and $j =1,\dots, n-k$, gives us the required diagonalization, which is an eigenvalue decomposition as \eqref{eq:onb} is an orthonormal basis. The values of the principal and Gaussian curvatures follow.
\end{proof}

The easiest way to calculate the third fundamental form is to get slightly ahead of our discussion and use the expression for Ricci curvature in Corollary~\ref{cor:Ricci} together with a result of Obata \cite[Theorem~1]{Obata68}. Otherwise we would have to start from the definition in Section~\ref{sec:ezoo}.
\begin{corollary}[Third fundamental form]\label{cor:third}
The third fundamental form $\tff_Q: \mathbb{T}_Q \Gr(k,n) \times \mathbb{T}_Q\Gr(k,n) \to  \mathbb{R}$ is given by
\[
\tff_Q(X,Y) =  \dfrac{1}{2 \Bigl( \dfrac{n}{2k(n-k)} - \dfrac{n-2}{4} \Bigr) \tr(X Y) =   \Bigl( \dfrac{n}{2k(n-k)} - \dfrac{n-2}{4} \Bigr) \tr(X_0^\tp Y_0),}
\]
with $Q$, $X$, $Y$ parameterized as in \eqref{eq:para}.
\end{corollary}
\begin{proof}
By \cite[Theorem~1]{Obata68},  we have 
\[
\tff_Q(X,Y) = \langle \sff_Q(X,Y),   \mathsf{H}_Q \rangle - \Ric(X,Y).
\]
By Theorem~\ref{thm:2ndFF}, Corollaries~\ref{cor:meancurv} and \ref{cor:Ricci},  we have 
\begin{align*}
\tff_Q(X,Y) 
&= \frac{1}{4k(n-k)}\bigl( (n-k)\tr(X_0 Y_0^\tp + Y_0 X_0^\tp) + k \tr (X_0^\tp Y_0 + Y_0^\tp X_0) \bigr) - \frac{n-2}{4}\tr(X_0^\tp Y_0) \\
&= \Bigl( \frac{n}{2k(n-k)} - \frac{n-2}{4} \Bigr) \tr(X_0^\tp Y_0). \qedhere
\end{align*}
\end{proof}

\section{Intrinsic curvatures of the Grassmannian}\label{sec:intrinsic}

As we will see in Section~\ref{sec:quot}, calculating intrinsic curvatures of Grassmannian with intrinsic geometry can get fairly involved. This is particularly striking for the Riemann curvature tensor ---  our calculation below is essentially one-line using the embedded geometry of the involution model.
\begin{proposition}[Riemmanian curvature]\label{prop:Rcurv}
The Riemann tensor $\Rie_Q: \mathbb{T}_Q \Gr(k,n) \times \mathbb{T}_Q \Gr(k,n) \times \mathbb{T}_Q \Gr(k,n) \times \mathbb{T}_Q\Gr(k,n) \to  \mathbb{R}$ is given by
\begin{align*}
 \Rie_Q(X,Y,Z,W)
& = \frac{1}{2}
\tr\bigl( (XY - YX) Z W \bigr) \\
&= \frac{1}{2} \tr\bigl( (X_0^\tp Y_0Z_0^\tp  + Z_0^\tp Y_0 X_0^\tp -  Y_0^\tp X_0 Z_0^\tp -  Z_0^\tp X_0 Y_0^\tp) W_0 \bigr)
\end{align*}
with $Q$, $X, Y, Z, W$  parameterized as in \eqref{eq:para}.
\end{proposition}
\begin{proof}
Using the expression for $\sff_Q$,
\begin{align*}
 \Rie_Q(X,Y,Z,W) &= \langle \sff_Q(Y, Z), \sff_Q(X, W) \rangle - \langle \sff_Q(X, Z), \sff_Q(Y, W) \rangle\\
&=\begin{multlined}[t] \frac{1}{4}\langle Y_0Z_0^\tp+Z_0Y_0^\tp, X_0W_0^\tp+W_0X_0^\tp\rangle + \frac{1}{4}\langle Y_0^\tp Z_0+Z_0^\tp Y_0, X_0^\tp W_0+W_0^\tp X_0\rangle \\
-\frac{1}{4}\langle X_0Z_0^\tp+Z_0X_0^\tp, Y_0W_0^\tp+W_0Y_0^\tp\rangle - \frac{1}{4}\langle X_0^\tp Z_0+Z_0^\tp X_0, Y_0^\tp W_0+W_0^\tp Y_0\rangle
\end{multlined} \\
&=\frac{1}{4}\langle [[X, Y], Z], W \rangle = \frac{1}{2}
\tr\bigl( (XY - YX) Z W \bigr),
\end{align*}
where the last equality is obtained by observing that $X,Y,Z,W$ are symmetric matrices.
\end{proof}

\begin{corollary}[Jacobi curvature]\label{cor:Jcurv}
The Jacobi tensor $\mathsf{J}_Q: \mathbb{T}_Q \Gr(k,n) \times \mathbb{T}_Q \Gr(k,n) \times \mathbb{T}_Q \Gr(k,n) \times \mathbb{T}_Q\Gr(k,n) \to  \mathbb{R}$ is
\begin{align*}
\mathsf{J}_Q(X, Y, Z, W) &= \tr(XYZW)  - \tr\Bigl(  Y \Bigl(\frac{X Z  + ZX}{2}\Bigr) W  \Bigr) \\
&=\begin{multlined}[t]
 \tr\bigl(  (X_0^\tp Y_0Z_0^\tp  + Z_0^\tp Y_0 X_0^\tp) W_0 \bigr) \\
 - \tr\Bigl( Y_0^\tp \Bigl(\frac{X_0 Z_0^\tp + Z_0X_0^\tp}{2} \Bigr) W_0 \Bigr) - \tr \Bigl( \Bigl( \frac{Z_0^\tp X_0 + X_0^\tp Z_0}{2} \Bigr) Y_0^\tp W_0 \bigr),
\end{multlined}
\end{align*}
with $Q$,  $X, Y, Z, W$  parameterized as in \eqref{eq:para}.
\end{corollary}
\begin{proof}
The expression in Proposition~\ref{prop:Rcurv} and the fact that $X,Y,Z,W$ are symmetric matrices yield
\[
\mathsf{J}_Q(X, Y, Z, W) = \frac{1}{2}\tr\bigl(  ( 2XYZ -   Y (X Z  + ZX  ) W  \bigr)
\]
and thus the first expression. Plugging in the parameterizations in \eqref{eq:para} for $X,Y,Z,W$ gives the second expression.
\end{proof}

\begin{corollary}[Sectional curvature]\label{cor:Scurv}
The sectional curvature $\upkappa_Q : \mathbb{T}_Q \Gr(k,n) \times \mathbb{T}_Q\Gr(k,n) \to  \mathbb{R}$ is given by
\[
\upkappa_Q(X,Y) = \frac{\|[X, Y]\|^2}{4(\lVert X \rVert^2 \lVert Y \rVert^2 - \langle X,  Y \rangle^2)} =  \frac{\lVert [X_0,  Y_0^\tp] \rVert^2 + 
\lVert [X_0^\tp,  Y_0] \rVert^2 
}{16(\lVert X_0 \rVert^2 \lVert Y_0 \rVert^2 - \langle X_0,  Y_0 \rangle^2)} \le \frac{1}{4},
\]
with $Q$, $X$, $Y$ parameterized as in \eqref{eq:para}. If $X, Y$ are orthonormal,  i.e.,  $\lVert X_0 \rVert = \lVert Y_0 \rVert = \sqrt{2}/2$ and 
$\langle X_0,  Y_0 \rangle = 0$,  then
\[
\upkappa_Q(X,Y) = \frac{\|[X, Y]\|^2}{4} =  \frac{1}{4} (\lVert [X_0,  Y_0^\tp] \rVert^2 + 
\lVert [X_0^\tp,  Y_0] \rVert^2).
\] 
\end{corollary}
\begin{proof}
This is a straightforward calculation by
\[
\upkappa_Q(X,Y) =  \frac{\Rie_Q(X,Y,Y,X)}{\lVert X \rVert^2 \lVert Y \rVert^2 - \langle X,  Y \rangle^2}
= \frac{1}{4}\frac{\langle [[X, Y], Y], X \rangle}{\lVert X \rVert^2 \lVert Y \rVert^2 - \langle X,  Y \rangle^2}
= \frac{\|[X, Y]\|^2}{4(\lVert X \rVert^2 \lVert Y \rVert^2 - \langle X,  Y \rangle^2)}
\]
and the observations that
\begin{gather*}
[X,  Y] = V \left(
\begin{bmatrix}
X_0 Y_0^\tp - Y_0 X_0^\tp  & 0 \\
0 & X_0^\tp Y_0 - Y_0^\tp X_0
\end{bmatrix}
\right) V^\tp,  \\
\lVert X \rVert^2 = 2 \lVert X_0 \rVert^2,\quad  \lVert Y \rVert^2 = 2 \lVert Y_0 \rVert^2,  \quad  \lVert \langle X,  Y \rangle \rVert^2 = 2 \langle X_0,  Y_0  \rangle. 
\end{gather*}
Since $\upkappa_Q(X,Y)$ only depends on the two-dimensional subspace of $\mathbb{T}_Q \Gr(k,n)$ spanned by $X$ and $Y$,  it suffices to assume that $X,Y$ are orthonormal.  The upper bound $\upkappa_Q(X,Y) \le 1/4$ then follows from the inequality $\lVert [A,  B] \rVert^2 \le 2 \lVert A \rVert^2   \lVert B \rVert^2$ for any $A, B \in \mathbb{R}^{n \times n}$.
\end{proof}

It is known that the Grassmannian manifolds are Einstein \cite[Paragraphs~0.25 and 0.26]{besse2007}. Our calculations below confirm the fact.
\begin{corollary}[Ricci and scalar curvatures]\label{cor:Ricci}
Let $Q$, $X$, $Y$ be parameterized as in \eqref{eq:para}. The Ricci tensor $\Ric_Q : \mathbb{T}_Q \Gr(k,n) \times \mathbb{T}_Q\Gr(k,n) \to  \mathbb{R}$  is given by 
\[
\Ric_Q(X,Y) = \frac{(n-2)}{8}\tr(X Y) = \frac{(n-2)}{4} \tr(X_0^\tp Y_0).
\]
The scalar curvature $\Sca_Q  \in \mathbb{R}$ is given by 
\[
\Sca_Q = \frac{k(n-k)(n-2)}{8}.
\]
The traceless Ricci curvature  $\mathsf{Z}_Q : \mathbb{T}_Q \Gr(k,n) \times \mathbb{T}_Q\Gr(k,n) \to  \mathbb{R}$ is given by 
\[
\mathsf{Z}_Q(X,Y) = 0,
\]
which shows that the Grassmannian is an Einstein manifold.
\end{corollary}
\begin{proof}
We write
\[
X_{ij} \coloneqq \frac{\sqrt{2}}{2}V \begin{bmatrix}
0 & E_{ij} \\
E_{ij}^\tp & 0
\end{bmatrix} V^\tp
\]
for the elements of the orthonormal basis in \eqref{eq:onb}. Then
\begin{align*}
\Ric_Q(X, Y) &= \sum_{i=1}^{k}\sum_{j=1}^{n-k} \Rie_Q(X_{ij},X,  Y, X_{ij}) = \sum_{i=1}^{k}\sum_{j=1}^{n-k}\frac{1}{4}\langle [[X_{ij}, X], Y], X_{ij} \rangle \rangle\\
&=\frac{1}{2}\sum_{i=1}^{k}\sum_{j=1}^{n-k}\tr( X_{ij}^2 XY - X X_{ij} Y X_{ij})\\
&=\frac{n-2}{8} \tr(X Y) = \frac{\Sca_Q}{k(n-k)} \mathsf{g}_Q(X, Y),
\end{align*}
where the last equality shows that $\mathsf{Z}_Q$ vanishes identically.
\end{proof}

For a homogeneous space like $\Gr(k,n)$,  the upper and lower delta invariants $\overline{\updelta}_{Q} (d_1,\dots,  d_r)$ and $\underline{\updelta}_{Q} (d_1,\dots,  d_r)$ are independent of the choice of $Q \in \Gr(k,n)$. We also restrict our attention to $d_1 = \cdots = d_r = 2$.  So for notational simplicity, we just write
\[
\overline{\updelta}_{2,r} \coloneqq \overline{\updelta}_{Q} (\underbrace{2,\dots,  2}_{\text{$r$ times}}),\quad \underline{\updelta}_{2,r} \coloneqq \underline{\updelta}_{Q} (\underbrace{2,\dots,  2}_{\text{$r$ times}}).
\]
\begin{theorem}[Delta invariants]\label{thm:delta}
Let $r \le  2 \lfloor k/2 \rfloor \lfloor (n-k)/2 \rfloor$. Then the upper and lower delta invariants of $\Gr(k,n)$ are given by
\[
\overline{\updelta}_{2,r} = \frac{k(n-k)(n-2)}{8},\qquad  \underline{\updelta}_{2,r} = \frac{k(n-k)(n-2)}{8} - \frac{r}{4}.
\]
\end{theorem}
\begin{proof}
We will write $\upkappa = \upkappa_Q$ below and  take $Q = I_{k,n-k}$.
By Corollary~\ref{cor:Ricci},  we have 
\begin{align*}
\overline{\updelta}_{2,r}  &= \frac{k(n-k)(n-2)}{8} - \inf_{\substack{\dim \mathbb{V}_j = 2,  \\ \mathbb{V}_j \perp \mathbb{V}_k,  j < k}} \biggl[ \sum_{j=1}^r  \upkappa(X_j, Y_j) \biggr],\\ 
\underline{\updelta}_{2,r}  &= \frac{k(n-k)(n-2)}{8} - \sup_{\substack{\dim \mathbb{V}_j = 2,  \\ \mathbb{V}_j \perp \mathbb{V}_k,  j < k}}   \biggl[ \sum_{j=1}^r \upkappa(X_j, Y_j) \biggr],
\end{align*}
where $\{ X_j,  Y_j \}$ is an orthonormal basis of the two-dimensional subspace $\mathbb{V}_j \subseteq \mathbb{T}_{I_{k,n-k}} \Gr(k,n)$,  $j =1,\dots, r$. By Corollary~\ref{cor:Scurv},  we have 
\begin{equation}\label{eq:del}
0 \le \sum_{j=1}^r \upkappa (X_j, Y_j) \le \frac{r}{4}.
\end{equation}
It remains to show that upper and lower bounds in \eqref{eq:del} are attained by some $\mathbb{V}_1,\dots,\mathbb{V}_r$.

Set $k_1 \coloneqq \lfloor k/2 \rfloor$ and $k_2 \coloneqq \lfloor (n-k)/2 \rfloor$. We may partition any $X \in \mathbb{T}_{I_{k,n-k}} \Gr(k,n)$ into a block matrix with $2 \times 2$ blocks $B_{pq} \in \mathbb{R}^{2\times 2}$:
\[
X =  \begin{bmatrix}
0 & \cdots & 0 & B_{1,1} & \cdots & B_{1,k_2+1} \\
\vdots & \ddots & \vdots  & \vdots & \ddots & \vdots \\ 
0 & \cdots & 0 & B_{k_1+1,1} & \cdots & B_{k_1+1,k_2+1} \\
B_{1,1}^\tp & \cdots & B_{k_1+1, 1}^\tp &  0 & \cdots & 0 \\
\vdots & \ddots & \vdots  & \vdots & \ddots & \vdots \\ 
B_{1, k_2+1}^\tp & \cdots & B_{k_1+1,k_2+1}^\tp  & 0 & \cdots & 0
\end{bmatrix}
\]
except in the last row and column where we are required to have
\[
B_{p,k_2 + 1}\in \mathbb{R}^{2\times (n-k - 2k_2)},  \quad B_{k_1 + 1, q} \in \mathbb{R}^{(k - 2k_1)\times 2}, \quad B_{k_1 + 1,  k_2 + 1} \in \mathbb{R}^{(k - 2k_1) \times (n-k - 2k_2)}
\]
for $p = 1,\dots, k_1$ and  $q=1,\dots, k_2$.

Let $\widehat{X}_{ij}  \in \mathbb{T}_{I_{k,n-k}} \Gr(k,n)$ be the tangent vector obtained from $X$ by setting $B_{pq} = 0$ whenever $(p,q) \ne (i,j)$.  Then clearly we have $\tr\bigl( \widehat{X}_{ij}^\tp \widehat{X}_{i'j'} \bigr) = 0$ whenever $(i,j ) \ne (i',j')$.  Since $r \le 2 k_1 k_2$,  the problem further reduces to attaining the upper and lower bounds in \eqref{eq:del} for $r = 2$ on $\Gr(2,4)$. This is a vast simplification as $X \in \mathbb{T}_{I_{2,2}}\Gr(2,4)$ is just
$X = \begin{bsmallmatrix}
0 & B \\
B^\tp & 0
\end{bsmallmatrix}$ with $B \in \mathbb{R}^{2\times 2}$. It remains to exhibit an orthonormal basis $ X_1, Y_1, X_2, Y_2 \in T_{I_{2,2}} \Gr(2,4)$ that gives the upper and lower bounds in \eqref{eq:del}. Using the formula for sectional curvature in Corollary~\ref{cor:Scurv}, we check that 
\begin{equation}\label{eq:upper bound}
X_{1} =\frac{\sqrt{2}}{2} \begin{bsmallmatrix}
0 & 0 & 1 & 0 \\
0 & 0 & 0 & 0 \\
1 & 0 & 0 & 0 \\
0 & 0 & 0 & 0 
\end{bsmallmatrix},\quad 
Y_{1} = \frac{1}{2} \begin{bsmallmatrix}
0 & 0 & 0 & 1 \\
0 & 0 & 1 & 0 \\
0 & 1 & 0 & 0 \\
1 & 0 & 0 & 0 
\end{bsmallmatrix},\quad
X_{2} =\frac{\sqrt{2}}{2} \begin{bsmallmatrix}
0 & 0 & 0 & 0 \\
0 & 0 & 0 & 1 \\
0 & 0 & 0 & 0 \\
0 & 1 & 0 & 0 
\end{bsmallmatrix},\quad
Y_{2} =\frac{1}{2} \begin{bsmallmatrix}
0 & 0 & 0 & 1 \\
0 & 0 & -1 & 0 \\
0 & -1 & 0 & 0 \\
1 & 0 & 0 & 0 
\end{bsmallmatrix}
\end{equation}
give the required upper bound $\upkappa(X_1, Y_1) + \upkappa(X_2, Y_2) = \frac{1}{2}$, whereas
\[
X_{1} =\frac{1}{2} \begin{bsmallmatrix}
0 & 0 & 1 & 0 \\
0 & 0 & 0 & 1 \\
1 & 0 & 0 & 0 \\
0 & 1 & 0 & 0 
\end{bsmallmatrix},\quad 
Y_{1} = \frac{1}{2} \begin{bsmallmatrix}
0 & 0 & 1 & 0 \\
0 & 0 & 0 & -1 \\
1 & 0 & 0 & 0 \\
0 & -1 & 0 & 0 
\end{bsmallmatrix},\quad
X_{2} =\frac{1}{2} \begin{bsmallmatrix}
0 & 0 & 0 & 1 \\
0 & 0 & 1 & 0 \\
0 & 1 & 0 & 0 \\
1 & 0 & 0 & 0 
\end{bsmallmatrix},\quad
Y_{2} =\frac{1}{2} \begin{bsmallmatrix}
0 & 0 & 0 & 1 \\
0 & 0 & -1 & 0 \\
0 & -1 & 0 & 0 \\
1 & 0 & 0 & 0 
\end{bsmallmatrix}
\]
give the required lower bound $\upkappa(X_1, Y_1) + \upkappa(X_2, Y_2) = 0$.
\end{proof}

We next compute the quartet of tensors named after Schouten, Cotton, Weyl, and Bach.
\begin{corollary}[Schouten curvature]\label{cor:Schouten}
The Schouten tensor $\mathsf{P}_Q : \mathbb{T}_Q \Gr(k,n) \times \mathbb{T}_Q\Gr(k,n) \to  \mathbb{R}$ is given by
\[
\mathsf{P}_Q(X,Y) = \frac{(n-2)}{16(k(n-k)-1)} \tr(X Y) = \frac{2(n-2)}{16(k(n-k)-1)} \tr(X_0^\tp Y_0)
\]
with $Q$, $X$, $Y$ parameterized as in \eqref{eq:para}.
\end{corollary}
\begin{proof}
This is a straightforward calculation from definition:
\begin{align*}
\mathsf{P}_Q(X, Y) &= \frac{1}{k(n-k)-2} \biggl[ \Ric_Q(X, Y) - \frac{\Sca_Q}{2(k(n-k)-1)} \mathsf{g}_Q(X, Y) \biggr]\\
&= \frac{n-2}{16(k(n-k)-1)} \tr(X Y). \qedhere
\end{align*}
\end{proof}

\begin{corollary}[Cotton curvature]\label{cor:Cotton}
The Cotton tensor of $\Gr(k,n)$ is zero.
\end{corollary}
\begin{proof}
By Corollary~\ref{cor:Schouten},  $\mathsf{P}$ is a constant multiple  of $\mathsf{g}$. So $\nabla \mathsf{P} = 0$, and so $\mathsf{C}$ is identically zero.
\end{proof}

\begin{corollary}[Weyl curvature]\label{cor:Weyl}
The Weyl tensor $\mathsf{W}_Q: \mathbb{T}_Q \Gr(k,n) \times \mathbb{T}_Q \Gr(k,n) \times \mathbb{T}_Q \Gr(k,n) \times \mathbb{T}_Q\Gr(k,n) \to  \mathbb{R}$ is given by
\begin{align*}
\mathsf{W}_Q (X,Y,Z,W) &=
\begin{multlined}[t]
\frac{1}{2} \tr \bigl( (XY - YX)Z W \bigr) \\
- \frac{(n-2)}{8(k(n-k) - 1)} \bigl( \tr(X Z) \tr(Y W)  - \tr(X W) \tr(Y Z) \bigr)
\end{multlined} \\
&= 
\begin{multlined}[t]
\frac12\tr\bigl( (X_0^\tp Y_0Z_0^\tp  + Z_0^\tp Y_0 X_0^\tp -  Y_0^\tp X_0 Z_0^\tp -  Z_0^\tp X_0 Y_0^\tp) W_0 \bigr)  \\
- \frac{n-2}{2 (k(n-k) - 1) } \bigl( \tr(X_0^\tp Z_0) \tr(Y_0^\tp W_0) - \tr(X_0^\tp W_0) \tr(Y_0^\tp Z_0) \bigr)
\end{multlined}
\end{align*}
with $Q$, $X$, $Y$,  $Z$, $W$ parameterized as in \eqref{eq:para}.
\end{corollary}
\begin{proof}
Let $m \coloneqq k (n-k)$.  It follows from the vanishing of $\mathsf{Z}_Q$ and the expression for $\Sca_Q$ in Corollary~\ref{cor:Ricci} that
\[
\mathsf{W}_Q =  \Rie_Q - \frac{1}{m-2} \mathsf{Z}_Q \varowedge \mathsf{g}_Q - \frac{\Sca_Q}{2m(m-1)} \mathsf{g}_Q \varowedge \mathsf{g}_Q
= \Rie_Q - \frac{n-2}{16(m-1)} \mathsf{g}_Q \varowedge \mathsf{g}_Q.
\]
Next use the two expressions of $\Rie_Q$ in Proposition~\ref{prop:Rcurv} and expand the Kulkarni--Nomizu product
\begin{align*}
\mathsf{g}_Q \varowedge \mathsf{g}_Q(X,Y,Z,W) &= 
\begin{multlined}[t]
\mathsf{g}_Q(X,Z) \mathsf{g}_Q(Y,W) - \mathsf{g}_Q(X,W) \mathsf{g}_Q(Y,Z) \\
- \mathsf{g}_Q(Y,Z) \mathsf{g}_Q(X,W) + \mathsf{g}_Q(Y,W) \mathsf{g}_Q(X,Z)
\end{multlined}\\
&=2 \bigl( \mathsf{g}_Q(X, Z) \mathsf{g}_Q(Y, W) - \mathsf{g}_Q(X, W) \mathsf{g}_Q(Y, Z) \bigr) \\
&= 8 \bigl( \tr(X_0^\tp Z_0) \tr(Y_0^\tp W_0) -  \tr(X_0^\tp W_0) \tr(Y_0^\tp Z_0) \bigr),
\end{align*}
to get the two required expressions for $\mathsf{W}_Q(X,Y,Z, W)$.
\end{proof}

\begin{corollary}[Bach curvature]\label{cor:Bach}
The Bach tensor $\mathsf{B}_Q : \mathbb{T}_Q \Gr(k,n) \times \mathbb{T}_Q\Gr(k,n) \to  \mathbb{R}$ is given by
\[
\mathsf{B}_Q(X,Y) = \frac{(n-2)^2}{32(k(n-k) - 2)} \tr(X Y) =  \frac{(n-2)^2}{16(k(n-k) - 2)} \tr(X_0^\tp Y_0)
\] 
with $Q$, $X$, $Y$ parameterized as in \eqref{eq:para}.
\end{corollary}
\begin{proof}
Let $m \coloneqq k (n-k)$.  Using \cite[Equation~(2-4)]{CH13}, we relate $\mathsf{B}$ to the Cotton tensor $\mathsf{C}$ as
\[
\mathsf{B}_Q(X,Y) = \frac{1}{m-2}\biggl[ \sum_{i=1}^m (\grad{X_i} \mathsf{C}_Q)(X_i,  X , Y) + \sum_{i=1}^m\sum_{j=1}^m \Ric_Q(X_i,  X_j) \mathsf{W}_Q(X,  X_i,  X_j,  Y) \biggr],
\]
where $X_1,\dots,X_m \in \mathbb{T}_{Q} \Gr(k,n)$ is any orthonormal basis. It follows from the vanishing of $\mathsf{C}$ in Corollary~\ref{cor:Cotton} that
\begin{align*}
\mathsf{B}_Q(X,Y) &= \frac{n-2}{8(m - 2)} \sum_{i=1}^m  \mathsf{W}_Q(X,  X_i,  X_i,  Y) \\
&= \frac{n-2}{8(m - 2)} \biggl[ \sum_{i=1}^m \Rie_Q(X, X_i , X_i, Y) - \sum_{i=1}^m \frac{n-2}{16(m-1)} \mathsf{g}_Q \varowedge \mathsf{g}_Q(X,  X_i,  X_i,  Y) \biggr] \\
&= \frac{n-2}{8(m - 2)}  \biggl[ \Ric_Q(X,Y)
 - \frac{2(n-2)}{16(m-1)} \sum_{i=1}^m ( \mathsf{g}_Q(X, X_i) \mathsf{g}_Q(Y,  X_i)  - \mathsf{g}_Q(X, Y) \mathsf{g}_Q(X_i,  X_i) )
 \biggr] \\
 &=  \frac{n-2}{8(m - 2)}  \Ric_Q(X,Y) - \frac{(n-2)^2}{64(m-1)(m-2)} (1 - m) \mathsf{g}_Q(X,Y) \\
 &=\frac{(n-2)^2}{32(m - 2)} \mathsf{g}_Q(X,Y).
\end{align*}
Here the first and the second equalities follow from Corollaries~\ref{cor:Ricci} and \ref{cor:Weyl} respectively. The third uses the definition of Ricci curvature and the value $\mathsf{g}_Q \varowedge \mathsf{g}_Q$ calculated in the proof of Corollary~\ref{cor:Weyl}. The penultimate equality is a result of $X_1,\dots,X_m$ being an orthonormal basis.
\end{proof}

\section{Geometric insights from these expressions}\label{sec:insight}

The intrinsic curvatures in Section~\ref{sec:intrinsic} are, by definition, independent of the model we choose and apply to $\Gr(k,\mathbb{R}^n)$ as an abstract manifold. Indeed, the results in this section will all be stated for the Grassmannian $\Gr(k,\mathbb{R}^n)$, as opposed to its involution model $\Gr(k,n)$. The expressions we found in Section~\ref{sec:intrinsic} by way of the involution model $\Gr(k,n)$ permit us to concretely study the geometry of the abstract manifold $\Gr(k, \mathbb{R}^n)$, and thereby obtaining new geometric insights.  Even without going out of our way to search for such insights, we can already see a few that, as far as we know,  have never been observed before for the Grassmannian.

For example, we may deduce the following, which is mildly surprising because we do not even know how to define a Pleba\'nski tensor  \cite{Pleb}.
\begin{corollary}[Pleba\'nski curvature]\label{cor:Pleb}
The Pleba\'nski tensor of the Grassmannian is zero.
\end{corollary}
An observant reader might have noticed that this is the only tensor mentioned in Section~\ref{sec:intro} whose definition did not appear in Section~\ref{sec:zoo}. The reason is that we do not know how to define the Pleba\'nski tensor in the coordinate-free manner adopted in modern mathematics. Every definition in the literature only gives its coordinates in terms of the coordinates of $\mathsf{Z}$, the traceless Ricci curvature. But as we do know from Corollary~\ref{cor:Ricci} that $\mathsf{Z} = 0$ for the Grassmannian, its Pleba\'nski tensor must be zero as well.

An observant reader might also have noticed that several of the expressions in Table~\ref{tab:c} are constant multiples of $\tr(XY)$. Therein lies two small results:
\begin{corollary}[Codazzi tensors I]\label{cor:Codazzi}
The Ricci, Schouten, and Bach curvatures of the Grassmannian are Codazzi.
\end{corollary}
\begin{proof}
Corollaries~\ref{cor:Ricci}, \ref{cor:Schouten}, and \ref{cor:Bach} show that the tensors in question are all constant multiples of $\mathsf{g}$ and therefore Codazzi since $\nabla \mathsf{g} = 0$.
\end{proof}
\begin{corollary}[Codazzi tensors II]
A symmetric bilinear form $\beta$ on the Grassmannian with constant trace is Codazzi if and only if $\nabla \beta = 0$.
\end{corollary}
\begin{proof}
If  $\nabla \beta = 0$, then it is Codazzi by definition. For the converse,  we invoke the result \cite{BE69} that any Codazzi tensor with constant trace on a compact Riemannian manifold  must have  $\nabla \beta = 0$ if the sectional curvature $\upkappa \ge 0$ everywhere. By Corollary~\ref{cor:Scurv}, the Grassmannian has $\upkappa \ge 0$.
\end{proof}

The proof of Corollary~\ref{cor:Codazzi} throws up another observation. 
\begin{corollary}[Divergence-free tensors]\label{cor:divfree}
The Riemann and Weyl curvatures of the Grassmannian are divergence-free.
\end{corollary} 
\begin{proof}
By Corollary~\ref{cor:Ricci}, $\Ric$ is a constant multiple of $\mathsf{g}$ and so  $\nabla \mathsf{\Ric} = 0$. Since the Riemann curvature is divergence-free if and only if the Ricci tensor is Codazzi \cite[Corollary~9.4.5]{Petersen16}, it follows from Corollary~\ref{cor:Codazzi} that $\dive \mathsf{Rie} = 0$ for the Grassmannian. The relation \cite[Equation~2-3]{CH13} between Weyl and Cotton tensors
\[
\dive \mathsf{W} = \frac{\dim \mathcal{M} - 3}{\dim \mathcal{M} - 2} \mathsf{C} 
\]
for any manifold $\mathcal{M}$ of dimension at least four, taken together with Corollary~\ref{cor:Cotton} that $\mathsf{C} = 0$, yields $\dive \mathsf{W} = 0$. For $\dim \mathcal{M} < 4$, $ \mathsf{W}$ is identically zero  \cite[Remark~2.3]{CH13}.
\end{proof} 

The delta invariants obtained in Theorem~\ref{thm:delta} have never before been calculated for a manifold as complex as the Grassmannian. These values may look quotidian to the uninitiated, but they are not. We give an example to show how the value of $\underline{\delta}_{2,r}$ found in Theorem~\ref{thm:delta} vastly improves a classical result.

A \emph{geodesic $2$-sphere} is a $2$-sphere $\mathrm{S}^2$ embedded in a Riemannian manifold $\mathcal{M}$ as a totally geodesic submanifold. One fascinating fact about the geometry of the Grassmannian is that it contains a geodesic $2$-sphere \cite{Led,Wang,Wolf,Wong68}. This is a very unique property. For instance,  $\mathbb{R}^3$ contains no geodesic $2$-sphere, even though $\mathrm{S}^2$ is, ironically, the unit sphere of $\mathbb{R}^3$. The reason is  that $\mathrm{S}^2$ is only a Riemannian submanifold but not a totally geodesic submanifold of $\mathbb{R}^3$.

A key result in \cite{Wong61} is that $\Gr(k,\mathbb{R}^n)$ contains \emph{one} geodesic $2$-sphere.  We will show in Theorem~\ref{thm:packing} that it in fact contains a product of \emph{many} geodesic $2$-spheres. To the best of our knowledge, this insight is new. It is also unusual. For example, the $3$-sphere $\mathrm{S}^3$ is known to contain a geodesic $2$-sphere but it does not contain a product of more than one copy.  Another example is the oriented Grassmannian $\oGr(2,\mathbb{R}^{n+2}) \cong \SO(n+2)/(\SO(n ) \times \SO(2))$, the set of oriented two-dimensional subspaces in $\mathbb{R}^{n+2}$. It is a close cousin of the Grassmannian, $\oGr(2, \mathbb{R}^{n+2})$ being a double cover of $\Gr(2,\mathbb{R}^{n+2})$. By \cite[Theorems~3.3 and 3.4]{CN77} and \cite[Section~5]{Klein08},  $\oGr(2,\mathbb{R}^{n+2}) $ contains $\mathrm{S}^2 \times \mathrm{S}^2$ as a totally geodesic submanifold if $n \ge 4$ but it cannot accommodate another extra copy, i.e., $\oGr(2,\mathbb{R}^{n+2}) $ does not contain $\mathrm{S}^2 \times \mathrm{S}^2 \times \mathrm{S}^2 $ as a totally geodesic submanifold  regardless of the value of $n$.  So the maximal dimension of a totally geodesic submanifold of the form $\mathrm{S}^2 \times \cdots \times \mathrm{S}^2$ in $\oGr(2,\mathbb{R}^{n+2})$ is independent of $n$. In contrast, this dimension increases at least linearly in $n$ for $\Gr(k, \mathbb{R}^n)$, as our next result shows.

\begin{theorem}[Embedding products of geodesic $2$-spheres]\label{thm:packing}
\begin{enumerate}[\upshape (a)]
\item \label{it:divf1} For any $r \le \min\{ \lfloor k/2 \rfloor,  \lfloor n/4 \rfloor \}$, the product of $r$ copies of $\mathrm{S}^2$ can be embedded as a totally geodesic submanifold of $\Gr(k, \mathbb{R}^n)$.
\item \label{it:divf2} For any $r \le 2\lfloor k/2 \rfloor \lfloor (n-k)/2 \rfloor$, an open subset of the product of $r$ copies of $\mathrm{S}^2$ can be embedded as a totally geodesic submanifold  of $\Gr(k, \mathbb{R}^n)$.
\end{enumerate}
\end{theorem}
\begin{proof}
Let $r \le \min\{ \lfloor k/2 \rfloor,  \lfloor n/4 \rfloor \}$ and $\mathbb{V}_1,  \dots,  \mathbb{V}_r \subseteq \mathbb{R}^n$ be any $r$ four-dimensional subspaces that are orthogonal to each other, i.e., $\mathbb{V}_j \subseteq \bigl( \bigoplus_{i \ne j} \mathbb{V}_i \bigr)^\perp$ for all $j = 1,\dots,r$. Let $\mathbb{W}_0 \subseteq  \bigl(\bigoplus_{i=1}^r \mathbb{V}_i \bigr)^\perp$ be a $(k - 2r)$-dimensional subspace. We claim that $\im(\varepsilon)$, the image of the embedding 
\[
\varepsilon :  \Gr(2,\mathbb{V}_1) \times \cdots \times \Gr(2,\mathbb{V}_r)\to \Gr(k,\mathbb{R}^n), \quad
(\mathbb{W}_1,\dots,  \mathbb{W}_r ) \mapsto  \mathbb{W}_0  \oplus \mathbb{W}_1 \oplus\dots \oplus \mathbb{W}_r,
\]
is totally geodesic:  Let $\mathbb{W} \coloneqq \mathbb{W}_0  \oplus \mathbb{W}_1 \oplus \dots \oplus \mathbb{W}_r$ and $\mathbb{W}' \coloneqq \mathbb{W}_0 \oplus \mathbb{W}_1'  \oplus \dots \oplus \mathbb{W}'_r \in \im(\varepsilon)$. The principal vectors \cite{GVL} between $\mathbb{W}$ and $\mathbb{W}'$ consist of any orthonormal basis of $\mathbb{W}_0$ together with principal vectors of $\mathbb{W}_i$ and $\mathbb{W}'_i$, $ i = 1,\dots, r$, noting that $\mathbb{W}_i \perp \mathbb{W}'_j$ whenever $i \ne j$. A geodesic in $\Gr(k,\mathbb{R}^n)$ connecting $\mathbb{W}$ and $\mathbb{W}'$ takes the form \cite{Wong67}
\[
\gamma(t) = \mathbb{W}_0 \oplus \mathbb{W}_1(t) \oplus \cdots \oplus \mathbb{W}_r(t) \in \im(\varepsilon),
\]
where $\mathbb{W}_i(t)$ is a geodesic in $\Gr(2,\mathbb{V}_i)$ connecting $\mathbb{W}_i$ to $\mathbb{W}'_i$, $i=1,\dots, r$. By \cite[Section~4]{Wong68}, since $\mathbb{V}_i$ is four-dimensional, $\Gr(2,\mathbb{V}_i)$ contains a geodesic $2$-sphere\footnote{When $\mathbb{V}$ is four-dimensional, any maximal subset of mutually isoclinic $2$-planes in $\Gr(2,\mathbb{V})$ is a geodesic $2$-sphere.} $\Sigma^2_i$ for each $i=1,\dots, r$.  The restriction of $\varepsilon$ to $\Sigma^2_1 \times \dots \times \Sigma^2_r \cong \mathrm{S}^2 \times \dots \times \mathrm{S}^2$ ($r$ copies) gives the desired embedding in \eqref{it:divf1}.

Our proof of \eqref{it:divf2} will exploit the matrix structure afforded by the involution model $\Gr(k,n) \cong \Gr(k, \mathbb{R}^n)$.  We will construct  a totally geodesic submanifold of $\Gr(k,n)$ containing $I_{k,n-k}$. By Corollary~\ref{cor:Scurv},  $\upkappa_{I_{k,n-k}}(X,Y) \le \frac14$ for any orthonormal $X, Y \in \mathbb{T}_{I_{k,n-k}} \Gr(k,  n)$.  By \cite[Theorem~5]{Wong68}, $\upkappa_{I_{k,n-k}}(X,Y) = \frac14$ if and only if  $X$ and $Y$ are tangent vectors of a geodesic $2$-sphere in $\Gr(k,n)$ passing through $I_{k,n-k}$. Note that $X$ and $Y$ must span the tangent space of this geodesic $2$-sphere. We set $k_1 \coloneqq  \lfloor k/2 \rfloor$, $k_2 \coloneqq  \lfloor (n-k)/2 \rfloor$ as in the proof of Theorem~\ref{thm:delta}, and also $r \coloneqq 2k_1 k_2$.  Consider the commutative diagram
\[\begin{tikzcd}
	{\varphi^{-1}(U)} & {\overbrace{\mathbb{R}^{2 \times 2} \oplus \cdots \oplus \mathbb{R}^{2\times 2}}^{r/2 \text{ copies}}} & {\mathbb{T}_{I_{k,n-k}} \Gr(k,n)} \\
	U & {\underbrace{ (\mathrm{S}^2 \times \mathrm{S}^2) \times \cdots \times (\mathrm{S}^2 \times \mathrm{S}^2)}_{r/2\text{ copies}}} & {\Gr(k,n)}
	\arrow[hook, from=1-1, to=1-2]
	\arrow["\rho", from=1-2, to=1-3]
	\arrow["{\varphi}"', from=1-2, to=2-2]
	\arrow["{\psi}", from=1-3, to=2-3]
	\arrow["{\varphi^{-1}}"', from=2-1, to=1-1]
	\arrow[hook, from=2-1, to=2-2]
	\arrow["{\psi \circ \rho \circ \varphi^{-1}}"', curve={height=12ex}, dashed, from=2-1, to=2-3]
\end{tikzcd}\]
where $\varphi,  \psi$ are the exponential maps on the respective tangent spaces,  $U$ is an open subset of $\mathrm{S}^2 \times \dots \times \mathrm{S}^2$ ($r$ copies) on which $\varphi^{-1}$ is well-defined, and $\rho$ is the linear map defined by
\[
B = (B_{ij})_{i,j=1}^{k_1,  k_2} \mapsto
\begin{bmatrix}
0 & \cdots & 0 & B_{1,1} & \cdots & B_{1,k_2+1}  \\
\vdots & \ddots & \vdots  & \vdots & \ddots & \vdots \\ 
0 & \cdots & 0 & B_{k_1+1,1} & \cdots & B_{k_1+1,k_2+1} \\
B_{1,1}^\tp & \cdots & B_{k_1+1, 1}^\tp &  0 & \cdots & 0 \\
\vdots & \ddots & \vdots  & \vdots & \ddots & \vdots \\ 
B_{1, k_2+1}^\tp & \cdots & B_{k_1+1,k_2+1}^\tp  & 0 & \cdots & 0
\end{bmatrix}
\]
where we have used the same notation as in the proof of Theorem~\ref{thm:delta} and set $B_{pq}$ to be the zero matrix if either $p = k_1 + 1$ or $q = k_2 + 1$.  We will write 
$\mathbb{M}$ for the vector space of block matrices in which $B$ lies, noting that $\mathbb{M} \cong \mathbb{R}^{2 \times 2} \oplus \cdots \oplus \mathbb{R}^{2\times 2}$ is a direct sum of $k_1k_2$ copies of $\mathbb{R}^{2\times 2}$, each corresponding to a block $B_{ij}$, $i =1,\dots,k_1$, $j=1,\dots,k_2$. For concreteness, we will also choose an orthonormal basis for each copy of $\mathbb{R}^{2\times 2}$ in $\mathbb{M}$,
\[
X_1 = \frac{\sqrt{2}}{2} \begin{bmatrix}
1 & 0 \\
0 & 0
\end{bmatrix},  \quad Y_1 = \frac{1}{2} \begin{bmatrix}
0 & 1 \\
1 & 0
\end{bmatrix},\quad X_2 =  \frac{\sqrt{2}}{2} \begin{bmatrix}
0 & 0 \\
0 & 1
\end{bmatrix},\quad Y_1 = \frac{1}{2} \begin{bmatrix}
0 & 1 \\
-1 & 0
\end{bmatrix},
\]
such that $\upkappa_{I_{k,n-k}}(\rho(X_1), \rho(Y_1)) = \upkappa_{I_{k,n-k}}(\rho(X_2), \rho(Y_2)) = \frac14$.  If $Z $ lies in the $(s,t)$th copy of $\mathbb{R}^{2\times 2}$ in $\mathbb{M}$, then let $\rho(Z) \coloneqq \rho\bigl( (B_{ij})_{i,j=1}^{k_1,k_2} \bigr)$ where $B_{ij} = \delta_{is} \delta_{jt} Z$,  $i=1,\dots, k _1$, $j = 1, \dots,k_2$. By shrinking $U$ if necessary,  we may assume that $\psi$ is injective on $\rho(\varphi^{-1}(U))$. Hence $\psi \circ \rho \circ \varphi^{-1}(U)$ is the required open subset in \eqref{it:divf2}.  For $r < 2 k_1 k_2$, let $x_0 \in \mathrm{S}^2$ be any fixed point and consider the embedding 
\[
j: \underbrace{\mathrm{S}^2 \times \dots \times \mathrm{S}^2}_{r \text{ copies}} \hookrightarrow  \underbrace{\mathrm{S}^2 \times \dots \times \mathrm{S}^2}_{2k_1 k_2 \text{ copies}},\quad j(x_1,\dots,  x_r) \coloneqq (x_1,\dots,  x_r, x_0,\dots,  x_0),
\]
where the last $2k_1 k_2 - r$ coordinates are all filled with $x_0$'s. Then $U' \coloneqq j^{-1}(U)$ is an open subset of $\mathrm{S}^2 \times \dots \times \mathrm{S}^2$ ($r$ copies) and $\psi\circ \rho \circ \varphi^{-1} \circ j(U')$ is a totally geodesic submanifold of $\Gr(k,n)$.
\end{proof}
Suppose $k$ and $n$ are both even. Then the upper bound in Theorem~\ref{thm:packing}\eqref{it:divf2} is $r = k(n-k)/2$. This is sharp as the dimension of the product of $r + 1$ copies of $\mathrm{S}^2$ is $k(n-k) + 2$ and it exceeds the dimension of $\Gr(k,n)$; so a product of $r+1$ copies of $\mathrm{S}^2$ cannot be embedded in $\Gr(k,n)$.

\section{Why we favor the involution model}\label{sec:other}

As we alluded to in Section~\ref{sec:intro}, a secondary goal of this article is to demonstrate the advantages of using the involution model \eqref{eq:inv}. Here we will make some comparisons with other common models of the Grassmannian in algebraic geometry (Section~\ref{sec:pluc}), differential geometry (Section~\ref{sec:quot}), and integral geometry (Section~\ref{sec:proj}).

To elaborate, as an abstract manifold, the Grassmannian $\Gr(k,\mathbb{R}^n)$ is just the set of $k$-planes in $\mathbb{R}^n$. While any manifold, by definition, can be given local coordinates, experience tells us that charts and atlases are not as useful as one might initially think. Especially in applied mathematics, but also in pure mathematics, the task at hand is often simplified when we give $\Gr(k,\mathbb{R}^n)$ a system of global, extrinsic coordinates --- this is what we mean by a \emph{model} for  $\Gr(k,\mathbb{R}^n)$.

\subsection{Pl\"ucker model}\label{sec:pluc}

The standard model of the Grassmannian in algebraic geometry (see \cite[Lecture~6]{Harris} and \cite[Chapter~1, Section~4.1]{Sha1}) is as the set of rank-one alternating tensors in projective space, i.e., the image of the Pl\"ucker embedding:
\[
\Gr(k, \mathbb{R}^n) \cong \bigl\{  [ v_1 \wedge \dots \wedge v_k ]  \in \mathbb{P}(\Skew^k( \mathbb{R}^n)) : v_1, \dots, v_k  \in \mathbb{R}^n \text{ linearly independent} \bigr\}.
\]
While this has some desirable mathematical properties \cite[Section~1]{LY24}, its main issue is that the ambient space $\mathbb{P}(\Skew^k( \mathbb{R}^n))$ is a manifold of  exceedingly high dimension $\binom{n}{k} - 1$.  This not only presents a computational conundrum but also results in complex expressions for even relatively basic quantities. For example, the second fundamental form has been derived in \cite[Lemma~2.1 and Proposition~2.3]{AG12} for the Pl\"ucker model and both its calculation and the expression are significantly more involved than those appearing in this article. In fact for even moderate values of $k$ the expressions in \cite{AG12} are next-to-impossible to use or even compute since they involve lengthy sums of high order tensors.

One observation from the extrinsic curvatures calculated in Section~\ref{sec:extrinsic} is that the involution model is extremely unlike the Pl\"ucker model. For example, the mean curvature of the image of the Pl\"ucker embedding is well-known to be zero but it is far from zero in the involution model, as we saw in Corollary~\ref{cor:meancurv}. Given that the mean curvature is determined by the second fundamental form, this shows that the second fundamental forms of both models must be different and therefore so are their Gaussian and principal curvatures.

\subsection{Quotient models}\label{sec:quot}

The most common models of the Grassmannian in differential geometry (see \cite[Volume II, Chapter~VII]{KN} and \cite[Chapter~9]{boothby}) are as one of several quotient spaces:
\begin{equation}\label{eq:quot}
\begin{aligned}
\Gr(k, \mathbb{R}^n) &\cong \quo{\O}{n}{k}  \\
&\cong \V(k,n) / \O(k) \\
&\cong \quo{\GL}{n}{k} \\
&\cong \St(k,n) / \GL(k),
\end{aligned}
\end{equation}
where $\V(k,n) \coloneqq \{V \in \mathbb{R}^{n \times k} :  V^\tp V = I \} $ and $\St(k,n) \coloneqq \{ X \in \mathbb{R}^{n \times k} : \rank(X) = k\}$ are respectively the usual models for the Stiefel manifolds of orthonormal $k$-frames and $k$-frames in $\mathbb{R}^n$, and $\O(n) = \V(n,n)$ and $\GL(n) = \St(n,n)$.

By exploiting their homogeneous space structures, the more basic intrinsic curvatures such as Riemann,  Ricci, and sectional curvatures of $\Gr(k, \mathbb{R}^n)$ are standard calculations that are classical in differential geometry \cite{Cartan46,  do1992, Samelson58}.  However, the use of quotient spaces inevitably gives rise to formulas involving horizontal lifts of tangent vectors and arbitrary representatives of equivalence classes. This introduces layer upon layer of ambiguities requiring multiple arbitrary choices.

We will walk the reader through the calculation of Riemann curvature in $\quo{\O}{n}{k}$ to illustrate the case in point. We will delimit equivalence classes by $\lb \,\cdot\, \rb$ below. In this model,
\begin{enumerate}[\upshape (i)]
\item a point $\lb Q \rb \in \quo{\O}{n}{k}$ is a coset 
\[
\lb Q \rb =
\left\lbrace Q \begin{bmatrix}
Q_1 & 0 \\
0 & Q_2
\end{bmatrix} \in \O(n) :  (Q_1, Q_2)  \in \O(k) \times \O(n - k)
\right\rbrace,
\]
for some $Q\in \O(n)$  but $Q$ is not canonically given;

\item a tangent vector $\lb X \rb_{\lb Q \rb} \in \mathbb{T}_{\lb Q \rb} \quo{\O}{n}{k}$ is an equivalence class of pairs
\[
\lb X \rb_{\lb Q \rb} = \bigl\{
\bigl(Q,  X + \sonk \bigr) \in \O(n) \times \sosonk
\bigr \}\!\! \Bigm/ \sim
\]
where the equivalence relation is defined by
\[
\bigl(Q,  X + \sonk \bigr) \sim \bigl( Q',  X' +  \sonk \bigr)
\]
if and only if there is some $ (Q_1, Q_2)  \in \O(k) \times \O(n - k)$ with
\begin{equation}\label{eq:rep}
\begin{aligned}
Q' &= Q \begin{bmatrix}
Q_1 & 0 \\
0 & Q_2
\end{bmatrix},\\
X' +  \sonk &= \begin{bmatrix}
Q_1 & 0 \\
0 & Q_2
\end{bmatrix}^\tp \bigl( X +  \sonk \bigr) \begin{bmatrix}
Q_1 & 0 \\
0 & Q_2
\end{bmatrix}.
\end{aligned}
\end{equation}
\end{enumerate}
Evidently, in this model even an object as basic as a tangent vector is an equivalence class (defined by $\sim$) of equivalence classes (the coset $X + \sonk$). Every layer of equivalence relations introduces a layer of ambiguity but more importantly it often takes additional effort in the form extra calculations  or computations.

Writing down a tangent vector $\lb X \rb_{\lb Q \rb} $ as a pair of actual matrices $(Q_0, X_0)$ requires making three arbitrary choices: first a representative $Q_0$ of $\lb Q \rb$, followed by a representative $(Q',  X' +  \sonk)$ of $\lb X \rb_{\lb Q \rb}$, and finally a representative $X_0$ of $X' +  \sonk$. Note that these cannot be chosen arbitrarily nor a priori but need to satisfy \eqref{eq:rep}. We will give the details below.

We begin by picking a representative $Q_0 \in \pi^{-1}(\lb Q \rb) \subseteq \O(n)$ where  $\pi: \O(n)  \to \quo{\O}{n}{k}$ is the quotient map,  a Riemannian submersion. To construct the horizontal lift $ H  \in \mathbb{T}_{Q_0} \O(n)$  of a tangent vector $\lb X \rb_{\lb Q \rb} \in \mathbb{T}_{\lb Q \rb} \quo{\O}{n}{k}$, the recommendation in the classic article of Edelman--Arias--Smith \cite{EAS99} is to use the isomorphism
\[
d_{Q_{0}} \pi: \left\lbrace Y\in \mathbb{R}^{n\times n}:   Q_{0}^\tp Y =  \begin{bmatrix}
0 & B \\
-B^\tp & 0
\end{bmatrix},\;  B\in \mathbb{R}^{k \times (n-k)} \right\rbrace 
\to \mathbb{T}_{\lb Q \rb} \quo{\O}{n}{k},
\]
defined for any  $Q \in \O(n)$ and  $\lb X \rb_{\lb Q \rb} \in \mathbb{T}_{\lb Q \rb} \quo{\O}{n}{k}$, and then compute the horizontal lift of $\lb X \rb_{\lb Q \rb}$ at $Q_0$ as $H \coloneqq (d_{Q_0} \pi)^{-1}(\lb X \rb_{\lb Q \rb})$.  

To get $H$ explicitly as a matrix, we will need to pick a representative $(Q',  X')$   for $\lb X \rb_{\lb Q \rb}$. To reduce notational clutter, we assume that $Q_0 = Q$.  We emphasize that in this problem:
\begin{quote}
The matrices $Q,  Q' \in \O(n)$ and $X' \in  \mathfrak{so}(n)$ are known but $X\in \mathfrak{so}(n)$ such that $(Q,X) \sim (Q',  X')$ is unknown and needs to be computed.
\end{quote}%
Observe that we cannot simply set $Q'$ to be $Q$ since $(Q, X')$ will not satisfy \eqref{eq:rep} in general.  Indeed, as we require
\[
\bigl(Q',   X' \bigr) \sim (Q,  \bigl( {Q'}^\tp  Q)^\tp  X' ({Q'}^\tp  Q) \bigr),
\]
we will need to compute $H$ as 
\[
H = Q( Q^{\prime\tp} Q)^\tp \begin{bmatrix}
0 & B' \\
-{B'}^\tp & 0
\end{bmatrix} (Q^{\prime\tp} Q)
= Q' \begin{bmatrix}
0 & B'  \\
-B^{\prime\tp} & 0
\end{bmatrix} Q^{\prime\tp} Q,
\]
with $B'$ the upper right $k\times (n-k)$ submatrix of $X'$.

It might appear that to compute the Riemann curvature\footnote{The expression comes from applying the standard method for calculating Riemann curvature on a quotient model of any symmetric space \cite[Theorem~4.2]{Helgasson01}.}
at $\lb Q \rb \in \quo{\O}{n}{k}$, we simply follow the procedure above to compute horizontal lifts $X, Y, Z, W \in \mathbb{T}_Q \O(n)$ of $\lb X \rb_{\lb Q \rb}, \lb Y \rb_{\lb Q \rb}, \lb Z \rb_{\lb Q \rb},  \lb W \rb_{\lb Q \rb} \in \mathbb{T}_{\lb Q \rb} \quo{\O}{n}{k}$ and evaluate the expression on the right hand-side:
\begin{equation}\label{eq:rie1}
\Rie (\lb X \rb, \lb Y \rb, \lb Z \rb,  \lb W \rb)   = -\frac{1}{4}\langle [[X, Y], Z], W \rangle.
\end{equation}
But this is a notational illusion. Even if we start from the same representative $Q \in \pi^{-1}(\lb Q \rb) \subseteq \O(n)$,  the procedure above will yield horizontal lifts of tangents vectors at \emph{different} representatives 
$Q_X', Q_Y', Q_Z', Q_W' \in \pi^{-1}(\lb Q \rb)$. There is no guarantee that $Q_X' = Q_Y' = Q_Z' = Q_W' = Q$ and extra steps are necessary to align these different representatives. To align $Q_X'$ with $Q$,  we need to compute the horizontal lift of $\lb X \rb_{\lb Q \rb}$ at $Q$ as
\[
(d_{Q} \pi)^{-1}(\lb X \rb_{\lb Q \rb})  = Q'_X (d_{Q_X'} \pi)^{-1}(\lb X \rb_{\lb Q \rb}) Q_X^{\prime\tp} Q.
\]
This process has to be repeated four times to get
\[
(d_{Q} \pi)^{-1}(\lb X \rb_{\lb Q \rb}),  \quad (d_{Q} \pi)^{-1}(\lb Y \rb_{\lb Q \rb}),  \quad (d_{Q} \pi)^{-1}(\lb Z \rb_{\lb Q \rb}),\quad  (d_{Q} \pi)^{-1}(\lb W \rb_{\lb Q \rb})
\] 
before we may evaluate \eqref{eq:rie1}.

Contrast this with  Proposition~\ref{prop:Rcurv}, where the calculation of Riemann curvature in the involution model avoids all these issues and the derivation of its expression is essentially a one-liner. It is important to point out that although the expression in \eqref{eq:rie1} superficially resembles our expression in Proposition~\ref{prop:Rcurv}, this is also a notational illusion --- they are completely different. The easiest way to see this is by observing that the matrices in \eqref{eq:rie1} are all skew-symmetric whereas those in Proposition~\ref{prop:Rcurv} are all symmetric.

The goal of Edelman, Arias, and Smith in \cite{EAS99} is to extend line search optimization methods to a function $f: \quo{\O}{n}{k} \to \mathbb{R}$. This permits them to work with one search direction, i.e., a single tangent vector $\lb X \rb_{\lb Q \rb}$, at every point $\lb Q \rb$. As a result, one could get around the problem by optimizing $f \circ \pi: \O(n)  \to \mathbb{R}$ along horizontal directions. In the calculation of curvatures, we  are required to work with four tangent vectors simultaneously and thus the alignment of different representatives of a given pair of equivalence classes cannot be avoided.

Although we have elected to make our point with $\quo{\O}{n}{k}$, the issues identified above apply to every quotient model in \eqref{eq:quot}. The root of these issues is that there is no global way to describe tangent vectors of $\Gr(k, \mathbb{R}^n)$ in any of these quotient models. Indeed, the absence of such a global description is the reason why expressions for various curvatures in \cite{Wong67,  Wong68,  Cartan46,  Samelson58,  do1992} can only be given locally. The same reason accounts for the expediency of the involution model --- not only does it describe all points on the Grassmannian \eqref{eq:inv}, it describes all tangent vectors at all point in a single unified way \eqref{eq:tangent}.

\subsection{Projection model}\label{sec:proj}

The standard model of the Grassmannian in integral geometry (see \cite[Chapter~9]{Nicolaescu} and \cite[Chapter~3]{Mattila}) is as the set of projection matrices:
\begin{equation}\label{eq:proj}
\Gr(k, \mathbb{R}^n) \cong \{P \in \mathbb{S}^n: P^2 = P,\; \tr(P) = k\}.
\end{equation}
This is the model closest to the involution model. Indeed we showed in   \cite[Theorem~6.1]{LY24}  that they are two instances in an infinite family of \emph{quadratic models} parameterized by the condition number of matrices used. However they are also on two opposite ends: The projection model is the unique model in this family that represents points as singular matrices (infinitely ill-conditioned)  whereas the involution model is the unique model in this family that represents points as orthogonal matrices (perfectly conditioned). Every other model in this family represents points with matrices of condition number strictly between one and infinity.

In numerical computations, methods based on projection matrices \cite[Lecture~8]{Trefethen} are well-known to be significantly less stable than methods based on orthogonal matrices \cite[Lecture~10]{Trefethen} --- in fact this comparison is famously used to illustrate numerical stability of algorithms.
In hand calculations, the singularity of projection matrices in \eqref{eq:proj} is a handicap, especially when contrasted against the ease of inverting the orthogonal matrices in \eqref{eq:inv}.

The equations defining \eqref{eq:proj} are also less convenient than those defining \eqref{eq:inv}. Any calculations involving tangent vectors in  the involution model would require one to differentiate $Q^2 = I$ to get  $X Q + Q X = 0$. But doing the same in the projection model would require one to differentiate $P^2 = P$ to get $XP + PX = X$. The latter is more difficult to use than the former; indeed any calculation involving the latter would usually involve a change of coordinates $P \mapsto 2P - I$ to simplify but that yields exactly the involution model since $Q = 2P - I$  \cite[Proposition~3.5]{ZLK20}. Nevertheless, because of this relation between the two models, every expression we derived for the involution model in this article gives one for the projection model, up to a constant factor.

The result \cite[Proposition~13]{FL19} comes close to obtaining the principal curvatures in  Corollary~\ref{cor:prin} using the projection model. Nevertheless, while \eqref{eq:proj} was used to model the manifold, the tangent spaces in \cite{FL19} were still modeled as horizontal spaces in the quotient model $\quo{\O}{n}{k}$, making the messy calculations in Section~\ref{sec:quot} all but unavoidable. 

Lastly the involution model has a unique feature that the projection and all other quadratic models lack, namely, it is naturally embedded in a Lie group as
\[
\{Q \in \mathbb{S}^n : Q^2 =I,\; \tr(Q)=2k - n\} = \{Q \in \O(n) : Q^\tp =Q,\; \tr(Q)=2k - n\}.
\]
In this article we have used the first embedding into $\mathbb{S}^n$ to perform our curvature calculations but the second embedding into $\O(n)$ is central to some of our other works \cite{ZLK20,prod}. Indeed, the involution model is the unique candidate for the missing entry in \cite[Table~2.1]{EAS99}, the matrix representation of Grassmannian that is currently marked as `none' in the table, as there is no other embedding of $\Gr(k, \mathbb{R}^n) $ into $\mathbb{R}^{n \times n}$ with orthogonality constraints \cite{LK24a,LK24b}.

\section{Conclusion}\label{sec:con}

In  studying curvatures, it is helpful to have an illuminating instance of a manifold $\mathcal{M}$ where all different forms of curvatures can be explicitly calculated and compared side-by-side like in Table~\ref{tab:c}. We are unaware of any nontrivial examples of this in the literature. The reason is clear in retrospect: For many of these curvatures, their defining equations in terms of local coordinates are near-impossible to calculate for anything more complex than a sphere. But even when embedded in $\mathbb{R}^m$, which one may do with Nash embedding, the resulting extrinsic coordinates are still difficult to use. The key to our simple formulas in this article is that we have embedded our manifold in a space of matrices, and matrices are endowed with far richer structures --- we may multiply or decompose them; impose orthogonality or symmetry on them; calculate their determinant, norm, or rank; find their eigen- or singular values and vectors; among a myriad of yet other features.

\subsection*{Future work}

The multitude of curvatures discussed in this article might lead the reader to think that we have exhausted the topic. This is not the case.

Some tensors are beyond our reach. The \emph{obstruction tensor} \cite[Equation~3.25]{FG} is a $2$-tensor that equals the Bach tensor in Corollary~\ref{cor:Bach} for four-dimensional manifolds but $\dim \Gr(k,n) \ne 4$ if $(k,n) \notin \{(1,5),(2,4)\}$. The \emph{Lanczos tensor} \cite{Lanc} is a $3$-tensor that is an antiderivative of the Weyl tensor in Corollary~\ref{cor:Weyl} and defined as a solution to a partial differential equation, which we are not even sure has a solution for $\Gr(k,n)$. The \emph{Bel tensor} \cite{Bel} is a $4$-tensor constructed by contracting the Riemann tensor with itself in all possible ways and the \emph{Bel--Robinson tensor} does the same with the Weyl tensor but we are unable to simplify their expressions for $\Gr(k,n)$ to the extent of those in Table~\ref{tab:c}.

Although we have limited the discussions in this article to the Levi-Civita connection, the natural choice from the perspective of Riemannian geometry, there are other alternatives if we study the Grassmannian in the context of non-Riemannian geometry.  In this case the torsion, nonmetricity, and cocurvature tensors discussed at the end of Section~\ref{sec:izoo} may no longer be zero. Using a different connection permits us to study yet other curvatures like the \emph{contorsion tensor} \cite[Theorem~6.2.5]{Bleecker}, a $3$-tensor that quantifies its deviation from Levi-Civita.

When presented with a complicated $d$-tensor $T \in  \mathbb{V}^{\otimes d}$, a common gambit in mathematics and physics \cite{Bel58, besse2007, Matte53, ST69, Strichartz88} is to decompose it by decomposing the space in which it lies. More precisely, for any compact Lie group $G\subseteq \GL(\mathbb{V})$, we decompose $ \mathbb{V}^{\otimes d} = \bigoplus_{\lambda \in \widehat{G}} \mathbb{V}_{\lambda} $ into irreducible $G$-submodules, giving a decomposition $T  = \sum_{\lambda \in \widehat{G}} T_\lambda$ with $T_{\lambda} \in \mathbb{V}_{\lambda}$. An example is the \emph{Ricci decomposition} of the Riemann curvature into the Weyl, traceless Ricci, and scalar curvatures \cite[Chapter~1, Section~G]{besse2007},
\[
\Rie  = \mathsf{W} + \frac{1}{k (n-k)-2} \mathsf{Z} \varowedge \mathsf{g} + \frac{\Sca}{2 k (n-k)\bigl( k (n-k)-1\bigr)} \mathsf{g} \varowedge \mathsf{g},
\]
which we have implicitly used in our definition of $\mathsf{W}$; here $d=4$, $\mathbb{V} = \mathbb{T}_Q \Gr(k,n)$, and $G = \O\bigl(k(n-k)\bigr)$. It would be interesting to find similar relations among the curvatures in Table~\ref{tab:c}.

Last but not least, while manifold optimization is not one of our goals here, it remains at the back of our minds. Existing optimization algorithms almost exclusively rely on two quantities --- gradient and Hessian. As shown in \cite{Yau}, it is certainly conceivable to use, say, the second fundamental form to optimize a smooth function. This and other curvatures computed in Table~\ref{tab:c} may turn out to be useful in this regard.

\subsection*{Acknowledgment} We thank the two anonymous reviewers for their exceptionally helpful comments and suggestions. ZL's work is partially supported by AFOSR MURI FA9550-19-1-0005. LH's work is partially supported by a Vannevar Bush Faculty Fellowship ONR N000142312863. KY's work is partially supported by the National Key R\&D Program of China 2023YFA1009401 and the National Natural Science Foundation of China 12288201. 

\bibliographystyle{abbrv}

\end{document}